%% file: main_arxiv.tex
\documentclass[11pt,a4paper]{amsart}
\usepackage[DIV14,paper=a4,BCOR14mm,headinclude]{typearea}
\usepackage[british]{babel}
\usepackage{amssymb}
\usepackage{graphicx}
\usepackage{color}
\usepackage{grffile}
\usepackage{tikz}
\usepackage{pgfplots}
\usepackage{enumitem}
\usepackage{hyperref}
\usepackage{fp-eval}
\usepackage{esint}
\usepackage{subfigure}
\usepackage{caption}
\usepackage{mathabx}
\usepackage{booktabs}
\usepackage{todonotes}
\definecolor{darkblue}{rgb}{0,0,.7}
\hypersetup{colorlinks=true,linkcolor=darkblue,citecolor=darkblue}

\newlist{alphenum}{enumerate}{1}
\setlist[alphenum]{fullwidth,label={(\alph*)}}

\allowdisplaybreaks[0]
\theoremstyle{definition}
\newtheorem{theorem}{Theorem}[section]

\newtheorem{remark}[theorem]{Remark}
\newtheorem{lemma}[theorem]{Lemma}

\numberwithin{figure}{section}
\numberwithin{table}{section}
\numberwithin{equation}{section}

\newcommand{\R}{\mathbb{R}}

\newcommand{\Div}{\mathrm{div}\;}

\newcommand{\vecb}[1]{\mathbf{#1}}

\begin{document}
\date{\today}

\title{Pressure-robustness in the context of optimal control}

\author{C.~Merdon}
\address{Christian Merdon, Weierstrass Institute for Applied Analysis
  and Stochastics, Mohrenstr. 39, 10117 Berlin, Germany,
  }
\email{Christian.Merdon@wias-berlin.de}
\author{W.~Wollner}
\address{Winnifried Wollner, Technische Universit{\"a}t Darmstadt,
  Fachbereich Mathematik, Dolivostr. 15, 64293 Darmstadt, Germany,
   }
\email{wollner@mathematik.tu-darmstadt.de}

\maketitle

\begin{abstract}
  This paper studies the benefits of pressure-robust discretizations in the scope of optimal control of incompressible flows. Gradient forces that may appear in the data can have a negative impact on the accuracy of state and control and can only be correctly balanced if their $L^2$-orthogonality onto discretely divergence-free test functions is restored. Perfectly orthogonal divergence-free discretizations or divergence-free reconstructions of these test functions do the trick and lead to much better analytic a priori estimates that are also validated in numerical examples.
\end{abstract}


\input{content.tex}

\section*{Acknowledgments}
We thank Alexander Linke for many pleasant and fruitful
discussions. W.~Wollner acknowledges funding by the Deutsche
Forschungsgemeinschaft (DFG, German Research Foundation) -- Projektnummer 392587580 -- SPP 1748


\input{main_arxiv.bbl}
\end{document}

%% file: content.tex
\section{Introduction}
The stationary Stokes equations seek an unknown velocity $\vecb{u}\in \vecb{V}$ and pressure $p \in Q$ such that
\begin{equation*}
  \begin{aligned}
    \nu (\nabla \vecb{u},\nabla \vecb{\varphi}) + (p,\Div
    \vecb{\varphi}) &=   (\vecb{f},\vecb{\varphi}) &
    \forall&\vecb{\varphi}\in \vecb{V}\\
    (\Div\vecb{u}, \psi) &= 0 & \forall&\psi \in Q
  \end{aligned}
\end{equation*}
with given data $\nu > 0$, $\vecb{f} \in \vecb{L}^2(\Omega)$ on a
domain $\Omega \subset \R^n$ with $n \le 3$, and the standard function
spaces
$\vecb{V} := \vecb{H}^1_0(\Omega;\R^d)$, and $Q := L^2_0(\Omega)$. A
standard discretization by inf-sup stable finite element spaces
$\vecb{V}_h \times Q_h \subset \vecb{V}\times Q$
yields solutions $(\vecb{u}_h,p_h)$ satisfying a best approximation
estimate, see, e.g.,~\cite[Section~III.1.2]{GiraultRaviart:1986}, or~\cite[Section~VI.2]{BrezziFortin:1991}, of the form
\[
  \| \vecb{u} -\vecb{u}_h \|^2_{\vecb{V}} \leq
  \frac{1}{\beta_h^2}\inf_{\vecb{\varphi}_h \in \vecb{V}_h}\|\vecb{u} -\vecb{\varphi}_h\|^2_V + C_P^2
  \quad \text{where} \quad
  C_P := \frac{1}{\nu} \inf_{\psi_h \in Q_h} \|p-\psi_h\|_Q
\]
with the discrete inf-sup constant $\beta_h$.
The estimate indicates, that the error in the velocity can be polluted
by a pressure which is hard to approximate. This is caused by a violation of the $L^2$-orthogonality between divergence-free functions and irrotational forces.
To remove this dependency, so called pressure-robust discretizations can be used, either perfectly orthogonal divergence-free methods \cite{SV:1983,Zhang:2011,FN:2013,GN:2014,Neilan:2015,JLMNR:2017} or modifications of classical methods via a reconstruction operator \(\Pi\) that repairs the orthogonality where needed \cite{Linke:2012,LMW:2015,LM:2016,LLMS:2017,JLMNR:2017,KVZ:2021}. With these methods estimates of the form
\[
  \| \vecb{u} -\vecb{u}_h \|^2_{\vecb{V}} \leq
  \frac{1}{\beta_h^2}\inf_{\vecb{\varphi}_h \in \vecb{V}_h}\|\vecb{u} -\vecb{\varphi}_h\|^2_V + C_\Pi^2
  \quad \text{where} \quad
  C_\Pi := \mathcal{O}(h^{k}) \|\Delta \vecb{u}\|_{H^{k-1}(\Omega)}
\]
are possible where the reconstruction causes a consistency error \(C_\Pi\) of optimal convergence order \(k\) (provided that \(\Delta \vecb{u} \in H^{k-1}\)) that is pressure-independent. Quasi-optimal estimates in case \(\Delta \vecb{u} \notin L^2(\Omega)\) can be found in \cite{LMN:2020}. Divergence-free $H^1$-conforming discretizations even come without any consistency error, i.e. \(C_\Pi= 0\), but usually require higher-order polynomial or non-standard ansatz spaces or specially refined meshes to ensure inf-sup stability. 
  Also note, that pressure-robust methods may have the potential to increase the accuracy beyond the presence of gradient errors in the data \(\vecb{f}\), but also in presence of complicated gradient forces generated by the material derivative in transient Navier-Stokes flows \cite{GLS:2019,ABBGLM:2021}.

This paper aims to investigate possible benefits of using pressure-robust discretizations in the context of
the optimization of incompressible flows, where a canonical
optimization problem is given by
\begin{equation}\label{eq:opt}\tag{P}
  \begin{aligned}
    \min_{(\vecb{q},\vecb{u},p) \in \vecb{Q} \times \vecb{V}\times Q}
    &\;\frac{1}{2}\|\vecb{u}-\vecb{u}^{\rm{d}}\|_{\vecb{L}^2(\Omega)}^2
    + \frac{\alpha}{2}\|\vecb{q}\|_{\vecb{L}^2(\Omega)}^2\\
    \text{s.t.}&\; \left\{
      \begin{aligned}
        \nu (\nabla \vecb{u},\nabla \vecb{\varphi}) + (p,\Div
        \vecb{\varphi}) &=   (\vecb{f}+\vecb{q},\vecb{\varphi}) &
        \forall&\vecb{\varphi}\in \vecb{V}\\
        (\Div\vecb{u}, \psi) &= 0 & \forall&\psi \in Q
      \end{aligned}\right.
    \end{aligned}
  \end{equation}
  with the control space $\vecb{Q} = \vecb{L}^2(\Omega;\R^n)$.
  
  The discretization of~\eqref{eq:opt} and similar control problems
  subject to an equation for an incompressible flow has been discussed
  extensively in the literature. 
  \cite{RoeschVexler:2006} showed optimal rates for a mixed finite
  element method for the above stationary Stokes control problem.
  Similar results on optimal convergence rates for the control
  of stationary and nonstationary Navier-Stokes control have been
  provided in~\cite{MR1079020} and~\cite{MR2050079}, respectively. 
  For least squares finite element approximations of the respective
  optimality system~\cite{MR2206446,MR2549788} showed best
  approximation results, and the same was done
  in~\cite{ChrysafinosKartzas:2015}
  for a standard Galerkin approximation of nonstationary Stokes control.
  For the related Dirichlet control problem error estimates have been
  obtained by HDG methods in~\cite{MR4169689} and for Navier-Stokes
  control in~\cite{MR1135991}

  All of the above results contain velocity errors depending on the
  pressure approximations. This implies that all the proposed methods
  will have spurious error contributions in the velocity induced by
  complicated pressures. In fact, irrotational forces can not only appear in the right-hand side
\(\vecb{f}\), but also in the data \(\vecb{u}^{\rm{d}}\). According to
its Helmholtz--Hodge decomposition \(\vecb{u}^{\rm{d}} = \nabla \psi +
\mathrm{curl} \phi\) only the divergence-free part \(\mathrm{curl}
\phi\) can be optimized, while the irrotational part \(\nabla \psi\)
cannot; but will confuse non-pressure-robust discretizations.
  Therefore, this paper discusses a
  pressure robust discretization of~\eqref{eq:opt}.
  A naive approach to pressure robustness in the control problem would
  be the use of a pressure robust discretization for the PDE
  constraint which we will detail in Section~\ref{sec:partialpressurerobust}.
  However, in view of irrotational forces in the data \(\vecb{u}^d\), this is insufficient and hence, Section~\ref{sec:fullpressurerobust}, provides a fully pressure robust approximation for~\eqref{eq:opt}.
  
\medskip
  The rest of the paper is structured as follows.
  Section~\ref{sec:canonical_opt_problem} introduces the canonical Stokes optimization problem studied in this paper.
  Section~\ref{sec:discretization_apriori} discusses
  the classical discretization and its a priori error estimate and suggests a novel pressure-robust discretization. Section~\ref{sec:analysis_fullyrobust} proves an a priori estimate for the fully pressure-robust variant that has qualitative improvements over the classical scheme. Section~\ref{sec:numerics} compares the three schemes
  in two numerical examples to illustrate the theoretical findings.

  \section{Canonical Stokes optimization problem}
  \label{sec:canonical_opt_problem}

Consider the optimal control problem: for given data \(\vecb{f} \in
\vecb{L}^2(\Omega)\), \(\vecb{u}^{\rm{d}} \in \vecb{L}^2(\Omega)\),
seek state and control \((\vecb{u},\vecb{q}) \in \vecb{V}^0 \times
\vecb{Q}\),
where $\vecb{V}^0 = \{ \vecb{\varphi} \in \vecb{V}\,|\, \Div
\vecb{\varphi} =0\} $ are the divergence-free functions, solving
\begin{align*}
  \min_{\substack{(\vecb{u},\vecb{q}) \in \vecb{V}^0 \times \vecb{Q}}}\;& \frac{1}{2}\|
\vecb{u} - \vecb{u}^{\rm{d}} \|^2_{\vecb{L}^2(\Omega)} +
\frac{\alpha}{2} \| \vecb{q} \|^2_{\vecb{L}^2(\Omega)}\\
&\text{s.t. } (\nu \nabla \vecb{u},\nabla \vecb{\varphi}) =
(\vecb{f}+\vecb{q},\vecb{\varphi})&\forall&\vecb{\varphi}\in \vecb{V}^0.
\end{align*}
Note, that this problem is equivalent to~\eqref{eq:opt} by
introducing a pressure $p\in Q$ to allow working with test functions
from $\vecb{V}$ rather than $\vecb{V}^0$.

Since this is a linear quadratic optimization problem, standard
theory, e.g.,~\cite{Troeltzsch:2010},
gives the necessary and sufficient optimality conditions with an
adjoint state $\vecb{w} \in \vecb{V}^0$ satisfying:
\begin{align*}
  0 & = \nu ( \nabla \vecb{u}, \nabla \vecb{\varphi}) -
  (\vecb{q} + \vecb{f}, \vecb{\varphi}) & \forall \vecb{\varphi}&\in\vecb{V}^0,\\
  0 & =   \nu (\nabla \vecb{\varphi},\nabla \vecb{w}) - (\vecb{u} - \vecb{u}^{\rm{d}}, \vecb{\varphi}) & \forall \vecb{\varphi}&\in\vecb{V}^0,\\
  0 & =  (\alpha \vecb{q} + \vecb{w}, \vecb{\varphi}) & \forall \vecb{\varphi}&\in\vecb{Q}.
\end{align*}
The third equation yields an algebraic relation 
\begin{align*}
  \vecb{q} = \frac{-1}{\alpha}\vecb{w} 
\end{align*}
between the adjoint $\vecb{w}$ and the control which could be used to
eliminate the control variable from the problem, by the so called
variational discretization approach~\cite{Hinze:2005}. 
However,~\cite{GaspozKreuzerVeeserWollner:2020} suggests that rather
than taking this simple substitution a more convenient choice for the
following analysis is the
consideration of the rescaled adjoint
\[
\vecb{z} = \frac{1}{\sqrt{\alpha}}\vecb{w}, \qquad \text{and hence}
\qquad \vecb{q} = \frac{-1}{\sqrt{\alpha}}\vecb{z}  . 
\]
From this it is easy to see that an optimal solution
$(\vecb{q},\vecb{u}) \in \vecb{Q}\times \vecb{V}^0$ of~\eqref{eq:opt} is equivalently given by a
solution $(\vecb{u},\vecb{z}) \in \vecb{V}^0 \times \vecb{V}^0$ of
\begin{equation}\label{eq:KKT}
  \begin{aligned}
    \nu ( \nabla \vecb{u}, \nabla \vecb{\varphi}) +
    \alpha^{-1/2} (\vecb{z},\vecb{\varphi})  &= (\vecb{f},
    \vecb{\varphi}) & \forall \vecb{\varphi}&\in\vecb{V}^0,\\
    \nu ( \nabla \vecb{\varphi}, \nabla \vecb{z}) -
    \alpha^{-1/2}(\vecb{u}, \vecb{\varphi}) &= -\alpha^{-1/2}(\vecb{u}^{\rm{d}}, \vecb{\varphi}) & \forall \vecb{\varphi}&\in\vecb{V}^0.\\
  \end{aligned}
\end{equation}
Adding the pressures or Lagrange multipliers for the divergence constraints, it is also equivalent to seek
$(\vecb{u},\vecb{z},p,\lambda) \in \vecb{V} \times \vecb{V} \times Q \times Q$
\begin{equation}\label{eq:KKTfull}
  \begin{aligned}
    \nu ( \nabla \vecb{u}, \nabla \vecb{\varphi})
    + (\Div\vecb{\varphi}, p) +
    \alpha^{-1/2} (\vecb{z},\vecb{\varphi})  &= (\vecb{f},
    \vecb{\varphi}) & \forall \vecb{\varphi}&\in\vecb{V},\\
    (\Div\vecb{u}, \psi)  &= 0 & \forall \psi &\in Q,\\
    \nu ( \nabla \vecb{\varphi}, \nabla \vecb{z})
    + (\Div\vecb{\varphi}, \lambda)  -
    \alpha^{-1/2}(\vecb{u}, \vecb{\varphi}) &= -\alpha^{-1/2}(\vecb{u}^{\rm{d}}, \vecb{\varphi}) & \forall \vecb{\varphi}&\in\vecb{V},\\
    (\Div\vecb{z}, \psi)  &= 0 & \forall \psi &\in Q.\\
  \end{aligned}
\end{equation}

\section{Discretization and a priori estimates}
  \label{sec:discretization_apriori}

\subsection{Preliminaries}

For the discretization the space \(\vecb{V}^0\) is replaced by some discretely divergence-free space
\begin{align*}
  \vecb{V}^0_h := \lbrace \vecb{\varphi}_h \in \vecb{V}_h : (\Div \vecb{\varphi}_h , \psi_h) = 0 \quad \text{for all } \psi_h \in Q_h\rbrace
\end{align*}
for some inf-sup stable pair \(\vecb{V}_h \times Q_h\).
The analysis involves the consistency error of the possibly relaxed divergence constraint in form of the dual norm
\begin{align}\label{eqn:dualnorm}
  \| \nabla \psi \|^2_{(\vecb{V}^0_h)^\star} := \sup_{\vecb{\varphi}_h \in \vecb{V}^0_h}
  \int \psi \, \Div(\vecb{\varphi}_h) \, dx / \| \nabla \vecb{\varphi}_h \|.
\end{align}
For exactly divergence-free schemes this norm is always zero, and otherwise can be bounded by the pressure best-approximation error \(\|\psi - \pi_{Q_h} \psi \|\).

Below we discuss a straight-forward classical discretization and some modified pressure-robust variant that replaces these errors by something qualitatively better.

It is studied how the consistency errors from the lack of pressure-robustness or
their pressure-robust alternatives influence the a priori estimates
for the natural energy norm induced by the PDE rather than the cost
functional, see
also~\cite{GaspozKreuzerVeeserWollner:2020},
\begin{align*}
|||(\vecb{u}, \vecb{z})|||^2 :=
   \| \nabla\vecb{u} \|^2_{L^2(\Omega)}
    + \| \nabla \vecb{z} \|^2_{L^2(\Omega)}.
\end{align*}
To do so, we estimate the distance of \((\vecb{u}_h, \vecb{z}_h)\) to the
Stokes best-approximations \((\mathbb{S}_h\vecb{u},
\mathbb{S}_h\vecb{z})\in \vecb{V}^0_h\times \vecb{V}^0_h\) ,
that are defined by
\begin{align}\label{eqn:stokes_projector}
  (\nabla(\vecb{w} - \mathbb{S}_h \vecb{w}), \nabla \vecb{\varphi}_h) = 0
  \quad \text{for all } \vecb{\varphi}_h \in \vecb{V}_h^0,
\end{align}
and allow for the Pythagoras theorem
\begin{equation}\label{eq:basicerror}
|||(\vecb{u} - \vecb{u}_h, \vecb{z} - \vecb{z}_h)|||^2
= |||(\vecb{u} - \mathbb{S}_h\vecb{u}, \vecb{z} - \mathbb{S}_h\vecb{z})|||^2
+ |||(\mathbb{S}_h\vecb{u} - \vecb{u}_h, \mathbb{S}_h\vecb{z} - \vecb{z}_h)|||^2.
\end{equation}
Due to inf-sup stability with inf-sup constant \(\beta_h > 0\), the first summand enjoys the
best-approximation property, i.e. for any \(\vecb{w} \in \vecb{V}\),
\begin{align*}
  \| \nabla( \vecb{w} - \mathbb{S}_h \vecb{w}) \|^2_{L^2(\Omega)}
  = \inf_{\vecb{\varphi}_h \in \vecb{V}_h^0} \| \nabla( \vecb{w} - \vecb{\varphi}_h) \|^2_{L^2(\Omega)} 
  \leq \frac{1}{\beta_h} \inf_{\vecb{\varphi}_h \in \vecb{V}_h} \| \nabla( \vecb{w} - \vecb{\varphi}_h) \|^2_{L^2(\Omega)}
\end{align*}
and shows convergence rates corresponding to the regularity of \(\vecb{u}\) and \(\vecb{z}\)
and the polynomial order of \(\vecb{V}_h\).
This best-approximation result is
only perturbed by the second summand $|||(\mathbb{S}_h\vecb{u} - \vecb{u}_h, \mathbb{S}_h\vecb{z} - \vecb{z}_h)|||^2$ which therefore is
the primal object of interest in the a priori error analysis below.

Moreover, assuming sufficient regularity of \(\vecb{w} \in \lbrace
\vecb{u}, \vecb{z}\rbrace\) and \(H^2\)-regularity of the
Stokes-problem on the domain $\Omega$,
the Stokes projection also enjoys the estimate
\begin{align}\label{eqn:l2estimate_stokesprojector}
\| \vecb{w} - \mathbb{S}_h \vecb{w} \|_{L^2(\Omega)}
\lesssim h \| \nabla (\vecb{w} - \mathbb{S}_h \vecb{w}) \|_{L^2(\Omega)}
\lesssim h^{k+1} \| \vecb{w} \|_{H^{k+1}}
\end{align}
which is needed in Theorem~\ref{lem:fullyrobust} below and
can be shown by the usual Aubin--Nitsche argument.

\subsection{Classical discretization}
The classical variational discretization of~\eqref{eq:KKTfull} solves the following discrete problem: seek \((\vecb{u}_h, \vecb{z}_h, p_h , \lambda_h) \in \vecb{V}_h \times \vecb{V}_h \times Q_h \times Q_h\) such that 
\begin{equation}\label{eqn:classical_discretisation}
\begin{aligned}
  \nu (\nabla \vecb{u}_h, \nabla \vecb{\varphi}_h) + (\Div
  \vecb{\varphi}_h,p_h)& = (\vecb{f} - \alpha^{-1/2}\vecb{z}_h, \vecb{\varphi}_h) & \forall&
  \vecb{\varphi}_h\in \vecb{V}_h,\\
  (\Div \vecb{u}_h , \psi_h) & = 0 &\forall&\psi_h \in Q_h,\\
  \nu (\nabla \vecb{\varphi}_h, \nabla \vecb{z}_h) + (\Div \vecb{\varphi}_h,\lambda_h) & = \alpha^{-1/2} (\vecb{u}_h-\vecb{u}^{\rm{d}}, \vecb{\varphi}_h)&\forall&
  \vecb{\varphi}_h\in \vecb{V}_h,\\
  (\Div \vecb{z}_h , \psi_h) & = 0&\forall&\psi_h \in Q_h.
 \end{aligned}
\end{equation}

The error estimates involve the previously defined Stokes projector
\(\mathbb{S}_h : \vecb{V} \rightarrow \vecb{V}^0_h\) and as stated above only the second summand in~\eqref{eq:basicerror} needs to be discussed.
\begin{lemma}[A priori error estimate for difference to
  best-approximation]
  \label{lem:classical}
  For the solution of \((\vecb{u}_h, \vecb{q}_h)\) of \eqref{eqn:classical_discretisation} and the discrete Stokes projectors of the exact solutions \((\mathbb{S}_h\vecb{u}, \mathbb{S}_h\vecb{q})\), it holds
  \begin{align*}
  |||(\mathbb{S}_h\vecb{u} - \vecb{u}_h, \mathbb{S}_h\vecb{z} - \vecb{z}_h)|||
    \leq&\; \frac{1}{\nu} \Bigl(\| \nabla p \|^2_{(\vecb{V}^0_h)^\star} + \| \nabla \lambda \|^2_{(\vecb{V}^0_h)^\star}\Bigr)^{1/2}\\ &+ \frac{1}{\alpha^{1/2}\nu} \Bigl(\| \vecb{u} - \mathbb{S}_h \vecb{u} \|^2 + \|\vecb{z} - \mathbb{S}_h \vecb{z}\|^2 \Bigr)^{1/2}.
  \end{align*}
  where \(p\) and \(\lambda\) are the respective pressures of the Stokes problems for \(\vecb{u}\) and \(\vecb{z}\). The dual norms, as defined in \eqref{eqn:dualnorm}, measure their generated consistency error due to the discrete divergence.
\end{lemma}
\begin{remark}
  Note that the upper bound
  \[
    \| \nabla \psi \|_{(\vecb{V}^0_h)^\star} \leq \| \psi - \pi_{Q_h} \psi \|_{L^2(\Omega)}
  \]
  together with assumed regularity of $p$ and $\lambda$
  would allow for the usual estimates in terms of powers of the mesh
  size.
  However, the upper bound \(\| \nabla \psi \|_{(\vecb{V}^0_h)^\star}\) is sharper, and in particular vanishes for
  divergence-free elements (i.e., $\vecb{V}_h^0 \subset
  \vecb{V}^0$). In this case only higher order terms according to \eqref{eqn:l2estimate_stokesprojector} remain on the right-hand side.
\end{remark}
\begin{proof}[Proof of Lemma~\ref{lem:classical}]
Using \eqref{eqn:stokes_projector} and
testing the first equation of \eqref{eq:KKTfull} and \eqref{eqn:classical_discretisation} with
\(\vecb{\varphi}_h = \mathbb{S}_h\vecb{u} - \vecb{u}_h\in \vecb{V}^0_h\) reveals
\begin{multline*}
\nu \| \nabla( \mathbb{S}_h\vecb{u} - \vecb{u}_h) \|^2_{L^2(\Omega)}\\
\begin{aligned}
& = \nu (\nabla(\vecb{u} - \vecb{u}_h), \nabla(\mathbb{S}_h\vecb{u} - \vecb{u}_h))\\
& = (\vecb{f} - \alpha^{-1/2}\vecb{z}, \mathbb{S}_h\vecb{u} - \vecb{u}_h) - (p,\Div(\mathbb{S}_h\vecb{u} - \vecb{u}_h))
- (\vecb{f} - \alpha^{-1/2}\vecb{z}_h, \mathbb{S}_h\vecb{u} - \vecb{u}_h)\\
& = -\alpha^{-1/2}(\vecb{z}-\vecb{z}_h, \mathbb{S}_h\vecb{u} - \vecb{u}_h) - (p,\Div(\mathbb{S}_h\vecb{u} - \vecb{u}_h)).
\end{aligned}
\end{multline*}
Analogously, one shows that
\begin{align*}
\nu \| \nabla( \mathbb{S}_h\vecb{z} - \vecb{z}_h) \|^2_{L^2(\Omega)}
& = \alpha^{-1/2} (\vecb{u}_h-\vecb{u}, \mathbb{S}_h\vecb{z} - \vecb{z}_h) - (\lambda,\Div(\mathbb{S}_h\vecb{z} - \vecb{z}_h)).
\end{align*}
Using
\begin{align*}
 (\vecb{z}-\vecb{z}_h, \mathbb{S}_h\vecb{u} - \vecb{u}_h)
 & = (\mathbb{S}_h\vecb{z} - \vecb{z}_h, \mathbb{S}_h\vecb{u} - \vecb{u}_h)
 + (\vecb{z} - \mathbb{S}_h \vecb{z}, \mathbb{S}_h\vecb{u} - \vecb{u}_h)\\
 (\vecb{u}-\vecb{u}_h, \mathbb{S}_h\vecb{z} - \vecb{z}_h)
 & = (\mathbb{S}_h\vecb{u} - \vecb{u}_h, \mathbb{S}_h\vecb{z} - \vecb{z}_h)
 + (\vecb{u} - \mathbb{S}_h \vecb{u}, \mathbb{S}_h\vecb{z} - \vecb{z}_h)
 \end{align*}
 one obtains
 \begin{align*}
 A :&= (\vecb{z}-\vecb{z}_h, \mathbb{S}_h\vecb{u} - \vecb{u}_h)-(\vecb{u}_h-\vecb{u}, \mathbb{S}_h\vecb{z} - \vecb{z}_h)\\
 & = (\vecb{z} - \mathbb{S}_h \vecb{z}, \mathbb{S}_h\vecb{u} - \vecb{u}_h)-(\vecb{u} - \mathbb{S}_h \vecb{u}, \mathbb{S}_h\vecb{z} - \vecb{z}_h)
\end{align*}
which can be estimated by
\begin{align*}
 |A| \leq
\left(\| \vecb{u} - \mathbb{S}_h \vecb{u} \|^2 + \|\vecb{z} - \mathbb{S}_h \vecb{z}\|^2\right)^{1/2} \times 
|||( \mathbb{S}_h\vecb{u} - \vecb{u}_h,  \mathbb{S}_h\vecb{z} - \vecb{z}_h) |||
\end{align*}
Analogously, for 
\begin{equation*}
  B := (p,\Div(\mathbb{S}_h\vecb{u} - \vecb{u}_h))  +  (\lambda,\Div(\mathbb{S}_h\vecb{z} - \vecb{z}_h)
\end{equation*}
one obtains the estimate
\begin{align*}
|B|  & \leq \Bigl(\| \nabla p \|^2_{(\vecb{V}^0_h)^\star} + \| \nabla \lambda \|^2_{(\vecb{V}^0_h)^\star}\Bigr)^{1/2} \times |||( \mathbb{S}_h\vecb{u} - \vecb{u}_h,  \mathbb{S}_h\vecb{z} - \vecb{z}_h) |||
\end{align*}
The summation of both estimates yields
\begin{align*}
  ||| ( \mathbb{S}_h\vecb{u} - \vecb{u}_h,\mathbb{S}_h\vecb{z} - \vecb{z}_h) |||^2 
   &=\; \frac{-1}{\nu} \Bigl(B + \alpha^{-1/2} A
  \Bigr)\\
  &\leq \nu^{-1} \left(\Bigl(\| \nabla p \|^2_{(\vecb{V}^0_h)^\star}
    + \| \nabla \lambda \|^2_{(\vecb{V}^0_h)^\star}\Bigr)^{1/2}
  \right.\\
  &\qquad\qquad+\left.    \alpha^{-1/2} \Bigl(\| \vecb{u} - \mathbb{S}_h \vecb{u} \|^2 +
    \|\vecb{z} - \mathbb{S}_h \vecb{z}\|^2 \Bigr)^{1/2} \right)\\
  &\quad\times |||( \mathbb{S}_h\vecb{u} - \vecb{u}_h,  \mathbb{S}_h\vecb{z} - \vecb{z}_h) |||.
\end{align*}
This concludes the proof.
\end{proof}

\subsection{Pressure-robust discretization}
In this section, we assume the existence of some reconstruction
operator \(\Pi : \vecb{V}_h+\vecb{V} \rightarrow \vecb{W}_h\) that maps
discretely divergence-free functions to exactly divergence-free
function, i.e., it holds
\[
  \Pi\colon \vecb{V}_h^0 \rightarrow \vecb{W}_h \cap \lbrace \vecb{\varphi} \in H(\mathrm{div},\Omega) : \Div\vecb{\varphi} = 0\rbrace.
\]
For the Bernardi--Raugel finite element methods used in the numerical examples one can use the standard interpolation \(\Pi = I_{\text{BDM}_1}\) into the Brezzi-Douglas-Marini space \(\vecb{W}_h := \mathrm{BDM}_1(\mathcal{T}) := \vecb{P}_1(\mathcal{T}) \cap H(\mathrm{div}, \Omega)\) where \(\vecb{P}_1\) denotes the piecewise affine vector-valued polynomials. This operator than has the property
\begin{align}\label{eqn:estimate_reconst}
  \| \vecb{\varphi} - \Pi \vecb{\varphi} \|_{L^2(\Omega)} \lesssim h^{m} \| \vecb{\varphi} \|_{H^{1+m}} \quad \text{for } m \in \lbrace 1,2 \rbrace
  \text{ and } \vecb{\varphi} \in H^m(\Omega)
\end{align}
which can be found in textbooks like \cite{BrezziFortin:1991}.
The Friedrichs inequality 
\(\| \vecb{\varphi} \|_{L^2(\Omega)} \leq C_F \| \nabla \vecb{\varphi}
\|_{L^2(\Omega)}\) then also implies the estimate
\begin{align}\label{eqn:stability_reconst}
  \| \Pi \vecb{\varphi} \|_{L^2(\Omega)}
  \leq \| \vecb{\varphi} - \Pi \vecb{\varphi} \|_{L^2(\Omega)} + \| \vecb{\varphi} \|_{L^2(\Omega}
  \lesssim (h + C_F) \| \nabla \vecb{\varphi} \|_{L^2(\Omega)}
   \quad \text{for any } \vecb{\varphi} \in \vecb{V}.
\end{align}
For higher-order finite elements ($k \geq 2$) the same property for \(m \in \lbrace 1 ,\ldots, k+1 \rbrace\) and additional orthogonality properties are needed such that, for any \(\vecb{g} \in H^{k-1}(\Omega)\), it holds
\begin{align*}
  (\vecb{g}, \vecb{\varphi} - \Pi \vecb{\varphi}) \lesssim h^{k} \| \vecb{g} \|_{H^{k-1}} \| \nabla \vecb{\varphi} \|
\end{align*}
to allow for an estimation of the consistency error by
\begin{align}\label{eqn:consistency_reconst}
\| \vecb{g} \circ (1 - \Pi) \|^2_{(\vecb{V}^0_h)^\star} & := \sup_{\vecb{\varphi}_h \in \vecb{V}^0_h} \frac{(\vecb{g},(1 - \Pi) \vecb{\varphi}_h)}{\|\nabla \vecb{\varphi}_h \|_{L^2(\Omega)}} \lesssim h^k \| \vecb{g} \|_{H^{k-1}}.
\end{align}
For more details and choices of reconstruction operators for higher order finite element methods see, e.g.,~\cite{LMT:2016,LM:2016,LLMS:2017}.

\subsubsection{Partially pressure-robust discretization}\label{sec:partialpressurerobust}
In the optimal control setting, a naive approach to pressure robustness would be the use of a pressure
robust discretization of the Stokes equation in~\eqref{eq:opt}, giving
the problem
\begin{align*}
  \min_{(\vecb{q}_h,\vecb{u}_h,p_h) \in \vecb{Q} \times \vecb{V}_h\times Q_h}
    &\;\frac{1}{2}\|\vecb{u}_h-\vecb{u}^{\rm{d}}\|_{\vecb{L}^2(\Omega)}^2
    + \frac{\alpha}{2}\|\vecb{q}_h\|_{\vecb{L}^2(\Omega)}^2\\
    \text{s.t.}&\; \left\{
      \begin{aligned}
        \nu (\nabla \vecb{u}_h,\nabla \vecb{\varphi}_h) + (p_h,\Div
        \vecb{\varphi}_h) &=   (\vecb{f}+ \vecb{q}_h,\Pi\vecb{\varphi}_h) &
        \forall&\vecb{\varphi}_h\in \vecb{V},\\
        (\Div\vecb{u}_h, \psi_h) &= 0 & \forall&\psi_h \in Q_h.
      \end{aligned}\right.
  \end{align*}
Here following the variational discretization approach
of~\cite{Hinze:2005} the control $\vecb{q}_h \in \vecb{Q}$ is only
discretized implicitly by the optimality conditions. In the case at
hand the optimality conditions yield $\vecb{q}_h = -\alpha^{-1/2}\Pi \vecb{z}_h  = -\alpha^{-1/2}\Pi \vecb{w}_h \in
\vecb{W}_h$.

Following the same arguments as in Section~\ref{sec:canonical_opt_problem} the
solution of this discretized optimization problem is given
equivalently by a solution 
\((\vecb{u}_h, \vecb{z}_h, p_h , \lambda_h) \in \vecb{V}_h \times \vecb{V}_h \times Q_h \times Q_h\) of
\begin{align*}
  \nu (\nabla \vecb{u}_h, \nabla \vecb{\varphi}_h) + (\Div
  \vecb{\varphi}_h,p_h)& = (\vecb{f} - \alpha^{-1/2} \Pi\vecb{z}_h, \Pi
  \vecb{\varphi}_h) & \forall &\vecb{\varphi}_h\in \vecb{V}_h\\
  (\Div \vecb{u}_h , \psi_h) & = 0& \forall&\psi_h \in Q_h\\
  \nu (\nabla \vecb{\varphi}_h, \nabla \vecb{z}_h) + (\Div \vecb{\varphi}_h,\lambda_h) & = \alpha^{-1/2} (\vecb{u}_h-\vecb{u}^{\rm{d}}, \vecb{\varphi}_h)& \forall &\vecb{\varphi}_h\in \vecb{V}_h\\
  (\Div \vecb{z}_h , \psi_h) & = 0& \forall&\psi_h \in Q_h.
\end{align*}
The optimality system shows that we can not expect a real advantage
of this partially pressure robust discretization compared to 
the classical formulation, since the adjoint still suffers from a lack
of pressure robustness and associated consistency errors for hidden
gradient fields in the data. That this is indeed the case is shown in
the numerical examples in Section~\ref{sec:numerics}.

\subsubsection{Fully pressure-robust
  discretization}\label{sec:fullpressurerobust}\ \\
To obtain a fully pressure-robust method, in addition to the Stokes equation
also the cost functional needs to be modified as follows
\begin{align*}
  \min_{(\vecb{q}_h,\vecb{u}_h,p_h) \in \vecb{Q} \times \vecb{V}_h\times Q_h}
    &\;\frac{1}{2}\|\Pi \vecb{u}_h-\vecb{u}^{\rm{d}}\|_{\vecb{L}^2(\Omega)}^2
    + \frac{\alpha}{2}\|\vecb{q}_h\|_{\vecb{L}^2(\Omega)}^2\\
    \text{s.t.}&\; \left\{
      \begin{aligned}
        \nu (\nabla \vecb{u}_h,\nabla \vecb{\varphi}_h) + (p_h,\Div
        \vecb{\varphi}_h) &=   (\vecb{f}+\vecb{q}_h,\Pi\vecb{\varphi}_h) &
        \forall&\vecb{\varphi}_h\in \vecb{V}_h,\\
        (\Div\vecb{u}_h, \psi_h) &= 0 & \forall&\psi_h \in Q_h.
      \end{aligned}\right.
  \end{align*}
Again, the optimization problem is equivalent to searching for a
solution of the reduced optimality system. Hence, we search
\((\vecb{u}_h, \vecb{z}_h, p_h , \lambda_h) \in \vecb{V}_h \times \vecb{V}_h \times Q_h \times Q_h\) solving
\begin{equation}\label{eqn:probust_discretisation}
\begin{aligned}
  \nu (\nabla \vecb{u}_h, \nabla \vecb{\varphi}_h) + (\Div \vecb{\varphi}_h,p_h)& = (\vecb{f} -\alpha^{-1/2}\Pi \vecb{z}_h, \Pi \vecb{\varphi}_h)& \forall &\vecb{\varphi}_h\in \vecb{V}_h,\\
  (\Div \vecb{u}_h , \psi_h) & = 0& \forall&\psi_h \in Q_h,\\
  \nu (\nabla \vecb{\varphi}_h, \nabla \vecb{z}_h) + (\Div
  \vecb{\varphi}_h,\lambda_h) & = \alpha^{-1/2} (\Pi \vecb{u}_h - \vecb{u}^{\rm{d}}, \Pi \vecb{\varphi}_h)& \forall &\vecb{\varphi}_h\in \vecb{V}_h,\\
  (\Div \vecb{z}_h , \psi_h) & = 0& \forall&\psi_h \in Q_h.
\end{aligned}
\end{equation}

\section{Analysis of the fully pressure-robust method}
\label{sec:analysis_fullyrobust}

\begin{lemma}[A priori error estimate]\label{lem:fullyrobust}
  For the solution \((\vecb{u}_h, \vecb{z}_h)\) of
  \eqref{eqn:probust_discretisation} and the discrete Stokes
  projectors \((\mathbb{S}_h\vecb{u}, \mathbb{S}_h\vecb{z})\) of the
  (assumed to be sufficiently smooth) exact solutions, it holds
  \begin{align*}
    |||(\mathbb{S}_h\vecb{u} - \vecb{u}_h, \mathbb{S}_h\vecb{z} - \vecb{z}_h)|||
    \leq&\; \left(\| \Delta \vecb{u} \circ (1 - \Pi) \|^2_{{\vecb{V}^0_h}^\star} + \| \Delta \vecb{z} \circ (1 - \Pi) \|^2_{{\vecb{V}^0_h}^\star}\right)^{1/2}\\
    &+ \frac{1}{\nu\alpha^{1/2}} \left(\| (1 - \Pi \mathbb{S}_h)
      \vecb{u} \circ \Pi \|^2_{(\vecb{V}^0_h)^\star}\right. \\
    &\qquad\qquad\quad+ \left.\| (1 - \Pi \mathbb{S}_h) \vecb{z} \circ \Pi \|^2_{(\vecb{V}^0_h)^\star} \right)^{1/2}.
  \end{align*}
  The consistency error caused by the reconstruction operators can be
  estimated under the assumption of the \(H^2\)-regularity of the
  Stokes operator on the given domain and~\eqref{eqn:estimate_reconst},~\eqref{eqn:consistency_reconst} for the
  reconstruction operator, by
  \begin{align*}
    \| \Delta \vecb{w} \circ (1 - \Pi) \|^2_{(\vecb{V}^0_h)^\star} & := \sup_{\vecb{\varphi}_h \in \vecb{V}^0_h} \frac{(\Delta \vecb{w},(1 - \Pi) \vecb{\varphi}_h)}{\|\nabla \vecb{\varphi}_h \|_{L^2(\Omega)}} \lesssim h^k \| \Delta \vecb{w} \|_{H^{k-1}},\\
   \| (1 - \Pi \mathbb{S}_h) \vecb{w} \circ \Pi \|^2_{(\vecb{V}^0_h)^\star} & := \sup_{\vecb{\varphi}_h \in \vecb{V}^0_h} \frac{((1 - \Pi \mathbb{S}_h) \vecb{w}, \Pi \vecb{\varphi}_h)}{\|\nabla \vecb{\varphi}_h \|_{L^2(\Omega)}}
    \lesssim h^{k+1} \| \vecb{w} \|_{H^k}.
  \end{align*}
\begin{remark}
  The second norm compares \(\vecb{w}\) with its reconstructed Stokes projection.
For the exactly divergence free Scott--Vogelius element it holds \(\Pi = 1\) and
hence one obtains the same estimate as in Lemma~\ref{lem:classical}.
\end{remark}
\end{lemma}
\begin{proof}[Proof of Lemma~\ref{lem:fullyrobust}]
Using \eqref{eqn:stokes_projector} and
testing the first equation of \eqref{eqn:probust_discretisation} with
\(\vecb{\varphi}_h = \mathbb{S}_h\vecb{u} - \vecb{u}_h\in \vecb{V}^0_h\), and using \(\vecb{f} = - \nu \Delta \vecb{u} - \nabla p + \alpha^{-1/2}\vecb{z}\)
and \((\nabla p, \Pi \vecb{\varphi}_h) = 0\) reveals
\begin{align*}
\nu \| \nabla( \mathbb{S}_h\vecb{u} - \vecb{u}_h) \|^2_{L^2(\Omega)}
 =&\; \nu (\nabla(\vecb{u} - \vecb{u}_h), \nabla(\mathbb{S}_h\vecb{u} - \vecb{u}_h))\\
 =&\; - \nu (\Delta \vecb{u}, \mathbb{S}_h\vecb{u} - \vecb{u}_h)
- (\vecb{f} -\alpha^{-1/2}\Pi \vecb{z}_h, \Pi(\mathbb{S}_h\vecb{u} - \vecb{u}_h))\\
 =&\; - \nu (\Delta \vecb{u}, (1-\Pi)(\mathbb{S}_h\vecb{u} - \vecb{u}_h))\\
&\;- \alpha^{-1/2}(\vecb{z} - \Pi \vecb{z}_h, \Pi(\mathbb{S}_h\vecb{u} - \vecb{u}_h)).
\end{align*}
Analogously, one obtains
\begin{align*}
\nu \| \nabla( \mathbb{S}_h\vecb{z} - \vecb{z}_h) \|^2_{L^2(\Omega)}
& = -\nu (\Delta \vecb{z}, (1-\Pi)(\mathbb{S}_h\vecb{z} - \vecb{z}_h))
+ \alpha^{-1/2} (\vecb{u} - \Pi \vecb{u}_h, \Pi(\mathbb{S}_h\vecb{z} - \vecb{z}_h)).
\end{align*}
Observe that
\begin{align*}
  - \nu (\Delta \vecb{u}, (1-\Pi)(\mathbb{S}_h\vecb{u} - \vecb{u}_h))
  &- \nu (\Delta \vecb{z}, (1-\Pi)(\mathbb{S}_h\vecb{z} - \vecb{z}_h))\\
  \leq&\; \nu \| \Delta \vecb{u} \circ (1 - \Pi) \|_{{\vecb{V}^0_h}^\star} \| \nabla( \mathbb{S}_h\vecb{u} - \vecb{u}_h) \|_{L^2(\Omega)}\\
  &+ \nu \| \Delta \vecb{z} \circ (1 - \Pi) \|_{{\vecb{V}^0_h}^\star} \| \nabla( \mathbb{S}_h\vecb{z} - \vecb{z}_h) \|_{L^2(\Omega)}\\
   \leq&\; \nu \left(\| \Delta \vecb{u} \circ (1 - \Pi)
    \|^2_{{\vecb{V}^0_h}^\star} +  \| \Delta \vecb{z} \circ (1 - \Pi)
    \|^2_{{\vecb{V}^0_h}^\star} \right)^{1/2} \\
  &\;\times |||( \mathbb{S}_h\vecb{u} - \vecb{u}_h,\mathbb{S}_h\vecb{z} - \vecb{z}_h)|||.
\end{align*}
Similar to the Galerkin case, we have the additional higher order terms
(with factor \(-\alpha^{-1/2}\))
\begin{align*}
A :=  (\vecb{u} - \Pi \vecb{u}_h, \Pi(\mathbb{S}_h\vecb{z} -
\vecb{z}_h)) - (\vecb{z} - \Pi \vecb{z}_h, \Pi(\mathbb{S}_h\vecb{u} - \vecb{u}_h)).
\end{align*}
Again, some manipulations reveal
\begin{align*}
(\vecb{z} - \Pi \vecb{z}_h, \Pi(\mathbb{S}_h\vecb{u} - \vecb{u}_h))
& = (\vecb{z} - \Pi\mathbb{S}_h \vecb{z}, \Pi(\mathbb{S}_h\vecb{u} - \vecb{u}_h))
+ (\Pi(\mathbb{S}_h \vecb{z} -\vecb{z}_h), \Pi(\mathbb{S}_h\vecb{u} - \vecb{u}_h))\\
(\vecb{u} - \Pi \vecb{u}_h, \Pi(\mathbb{S}_h\vecb{z} - \vecb{z}_h))
& = (\vecb{u} - \Pi\mathbb{S}_h \vecb{u}, \Pi(\mathbb{S}_h\vecb{z} - \vecb{z}_h))
+ (\Pi(\mathbb{S}_h \vecb{u} -\vecb{u}_h), \Pi(\mathbb{S}_h\vecb{z} - \vecb{z}_h))
\end{align*}
and the subtraction of both lines leads to
\begin{align*}
|A| \leq&\; \left( \| (1 - \Pi \mathbb{S}_h) \vecb{u} \circ \Pi \|^2_{(\vecb{V}^0_h)^\star} + \| (1 - \Pi \mathbb{S}_h) \vecb{z} \circ \Pi \|^2_{(\vecb{V}^0_h)^\star}\right)^{1/2}\\
&\;\times ||| ( \mathbb{S}_h\vecb{u} - \vecb{u}_h, \mathbb{S}_h\vecb{z} - \vecb{z}_h) |||.
\end{align*}
The combination of these estimates concludes the proof of the
first claim and it remains to show the bounds for the consistency errors.
The first bound follows from \eqref{eqn:consistency_reconst} and the second bound can be estimated as follows.
The stability estimate
\eqref{eqn:stability_reconst} and triangle inequalities yield
\begin{align*}
    \sup_{\vecb{\varphi}_h \in \vecb{V}^0_h} \frac{( \vecb{w} - \Pi \mathbb{S}_h \vecb{w}, \Pi \vecb{\varphi}_h)}{\|\nabla \vecb{\varphi}_h \|_{L^2(\Omega)}}
    & \lesssim \|  \vecb{w} - \Pi \mathbb{S}_h \vecb{w} \|_{L^2}\\
    & \leq \|  \vecb{w} - \mathbb{S}_h \vecb{w} \|_{L^2}
    + \| \vecb{w} - \Pi \vecb{w} \|_{L^2}    
    + \| (1- \Pi) (\vecb{w} - \mathbb{S}_h \vecb{w}) \|_{L^2}\\
    & \lesssim \|  \vecb{w} - \mathbb{S}_h \vecb{w} \|_{L^2}
    + \| \vecb{w} - \Pi \vecb{w} \|_{L^2}    
    + h \| \nabla (\vecb{w} - \mathbb{S}_h \vecb{w}) \|_{L^2}. \end{align*}
The claimed estimate now follows from \eqref{eqn:estimate_reconst} and \eqref{eqn:l2estimate_stokesprojector}.
\end{proof}

\section{Numerical examples}\label{sec:numerics}\ \\
This section visualizes the theoretical results in two numerical examples that were conducted with the open source Julia package  \texttt{GradientRobustMultiPhysics.jl}.

To distinguish all three schemes, a common formulation is given by
\begin{align*}
  \min_{(\vecb{q}_h,\vecb{u}_h,p_h) \in \vecb{Q} \times \vecb{V}_h\times Q_h}
    &\;\frac{1}{2}\|\Pi_1 \vecb{u}_h-\vecb{u}^{\rm{d}}\|_{\vecb{L}^2(\Omega)}^2
    + \frac{\alpha}{2}\|\vecb{q}_h\|_{\vecb{L}^2(\Omega)}^2\\
    \text{s.t.}&\; \left\{
      \begin{aligned}
        \nu (\nabla \vecb{u}_h,\nabla \vecb{\varphi}_h) + (p_h,\Div
        \vecb{\varphi}_h) &=   (\vecb{f}+\vecb{q}_h,\Pi_2\vecb{\varphi}_h) &
        \forall&\vecb{\varphi}_h\in \vecb{V}_h,\\
        (\Div\vecb{u}_h, \psi_h) &= 0 & \forall&\psi_h \in Q_h,
      \end{aligned}\right.
  \end{align*}
where the classical scheme employs \(\Pi_1/\Pi_2 =
\mathrm{id}/\mathrm{id}\), the partially pressure-robust scheme
employs \(\Pi_1/\Pi_2 = \mathrm{id}/\Pi\) and the new fully pressure-robust scheme employs \(\Pi_1/\Pi_2 = \Pi/\Pi\).

\begin{figure}
\flushleft
\begin{tabular}{lll}
\fbox{$\alpha = 10^{-6}, \nu = 1\phantom{0^{-0}}$} &
\hspace{21ex}					   &
\fbox{$\alpha = 10^{-1}, \nu = 1\phantom{0^{-0}}$} 
\end{tabular}

\includegraphics[width=0.27\textwidth]{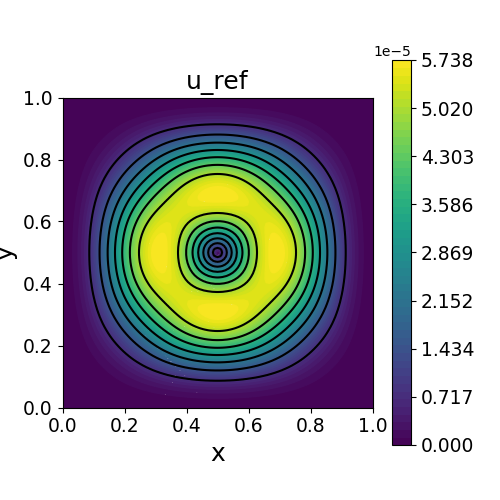}
\hspace{-2ex}
\includegraphics[width=0.22\textwidth, trim=5mm 0mm 0mm 0mm, clip]{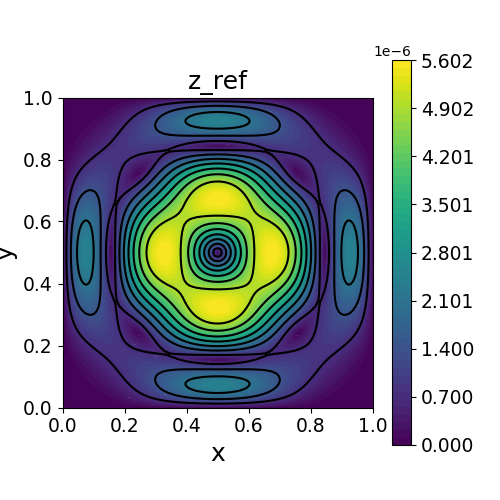}
\includegraphics[width=0.27\textwidth]{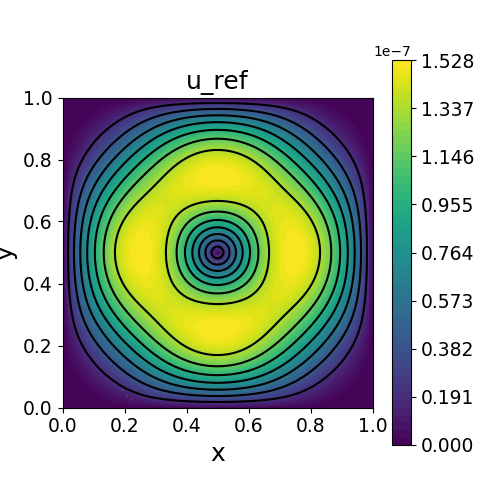}
\hspace{-2ex}
\includegraphics[width=0.22\textwidth, trim=5mm 0mm 0mm 0mm, clip]{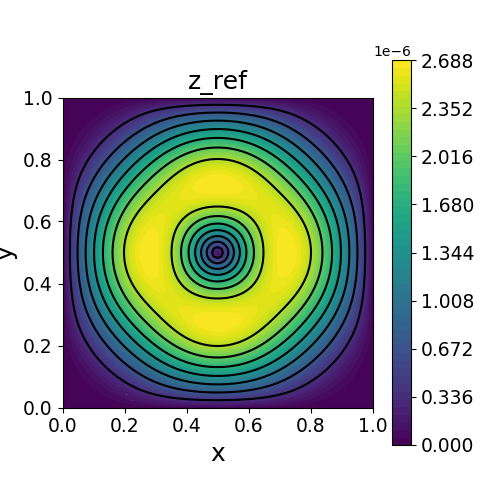}

\begin{tabular}{lll}
\fbox{$\alpha = 10^{-6}, \nu = 10^{-3}$} &
\hspace{21ex}					   &
\fbox{$\alpha = 10^{-1}, \nu = 10^{-3}$} 
\end{tabular}

\includegraphics[width=0.27\textwidth]{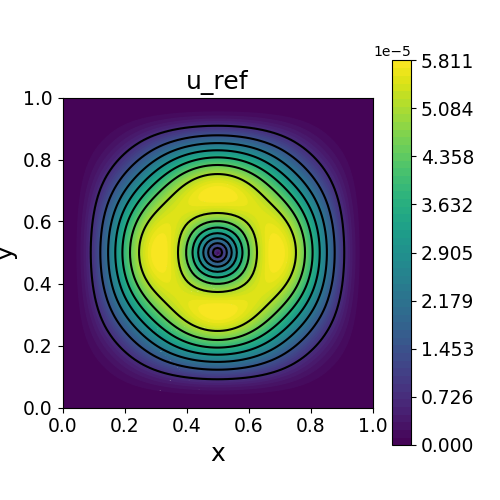}
\hspace{-2ex}
\includegraphics[width=0.22\textwidth, trim=5mm 0mm 0mm 0mm, clip]{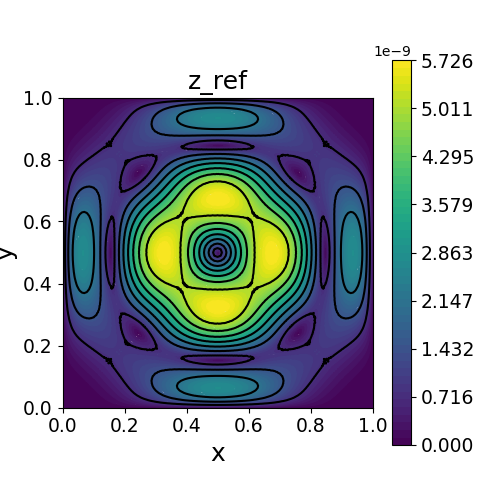}
\includegraphics[width=0.27\textwidth]{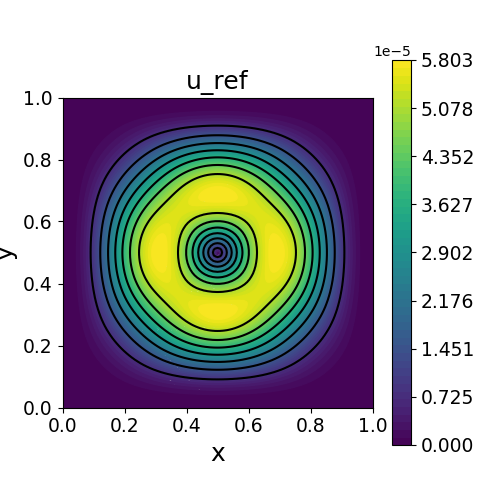}
\hspace{-2ex}
\includegraphics[width=0.22\textwidth, trim=5mm 0mm 0mm 0mm, clip]{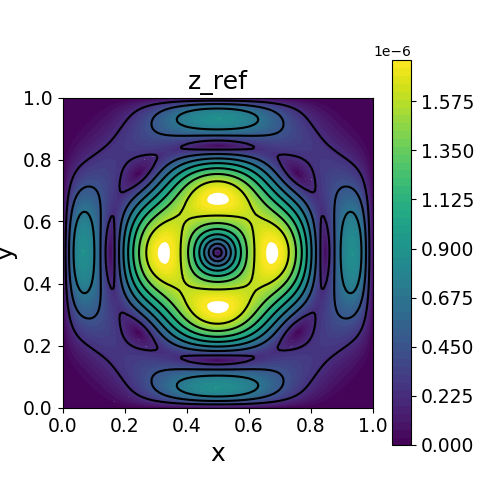}

\caption{\label{fig:ex1_refsolutions}Example 1: Reference solutions for \(\vecb{u}\) (larger images) and \(\vecb{z}\) (smaller images) for \(\nu = 1\) (top row) and \(\nu = 10^{-3}\) (bottom row) and \(\alpha = 10^{-6}\) (left column) and \(\alpha = 10^{-1}\) (right column), and independent of \(\epsilon\).}
\end{figure}

\begin{figure}
\flushleft
\begin{tabular}{lll}
\fbox{$\alpha = 10^{-1}$, classical} &
\hspace{24ex}					   &
\fbox{$\alpha = 10^{-1}$, fully p-robust} 
\end{tabular}
\includegraphics[width=0.27\textwidth]{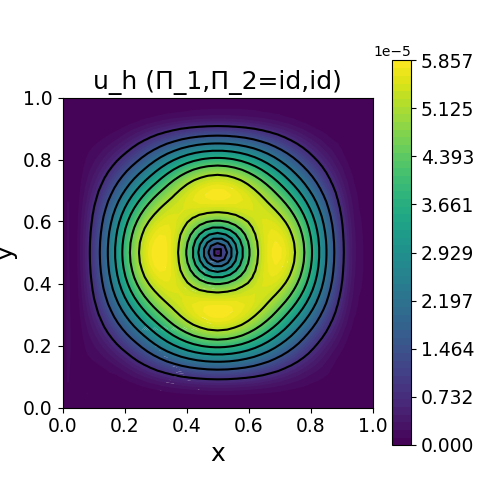}
\hspace{-2ex}
\includegraphics[width=0.22\textwidth, trim=5mm 0mm 0mm 0mm, clip]{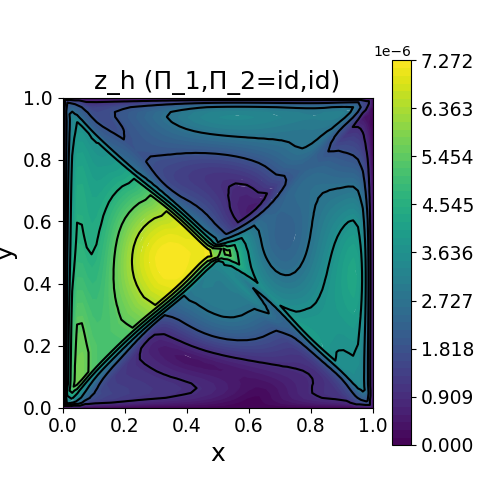}
\hfill
\includegraphics[width=0.27\textwidth]{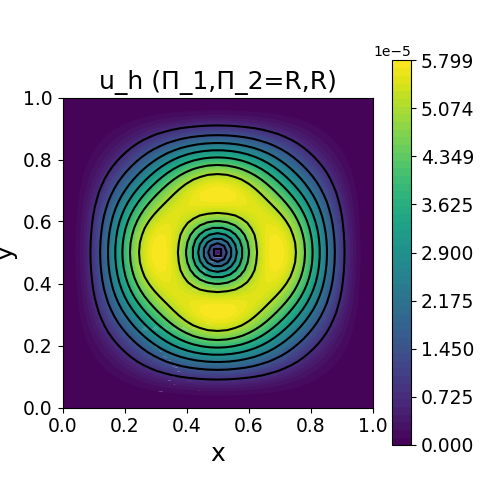}
\hspace{-2ex}
\includegraphics[width=0.22\textwidth, trim=5mm 0mm 0mm 0mm, clip]{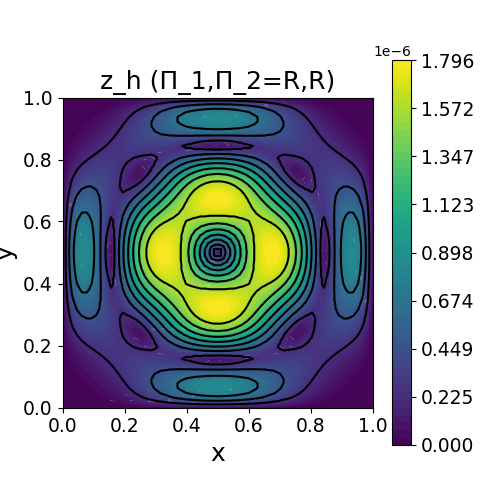}

\begin{tabular}{lll}
\fbox{$\alpha = 10^{-3}$, classical} &
\hspace{24ex}					   &
\fbox{$\alpha = 10^{-3}$, fully p-robust} 
\end{tabular}
\includegraphics[width=0.27\textwidth]{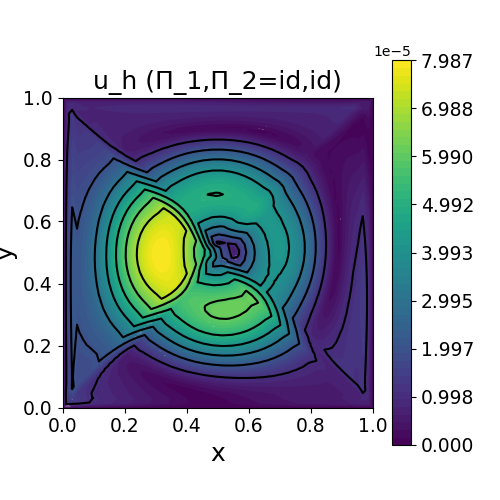}
\hspace{-2ex}
\includegraphics[width=0.22\textwidth, trim=5mm 0mm 0mm 0mm, clip]{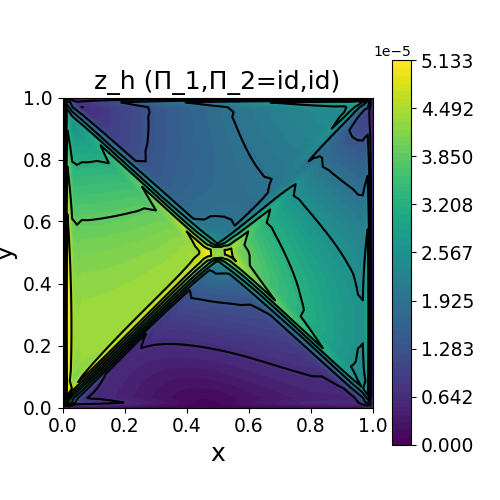}
\hfill
\includegraphics[width=0.27\textwidth]{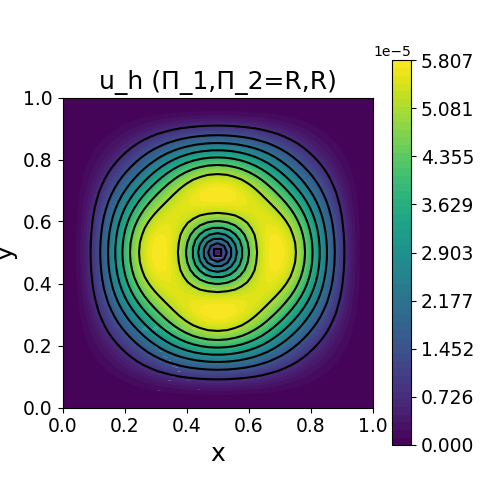}
\hspace{-2ex}
\includegraphics[width=0.22\textwidth, trim=5mm 0mm 0mm 0mm, clip]{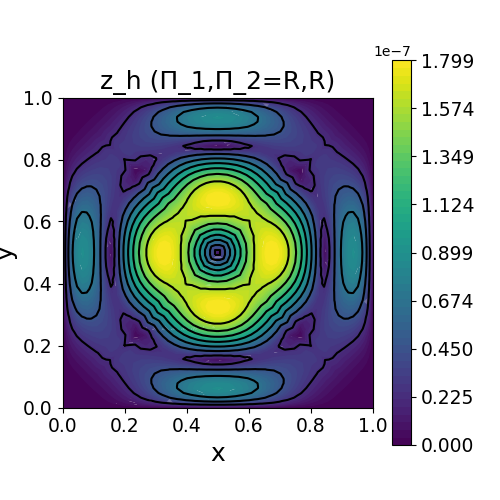}

\begin{tabular}{lll}
\fbox{$\alpha = 10^{-6}$, classical} &
\hspace{24ex}					   &
\fbox{$\alpha = 10^{-6}$, fully p-robust} 
\end{tabular}
\includegraphics[width=0.27\textwidth]{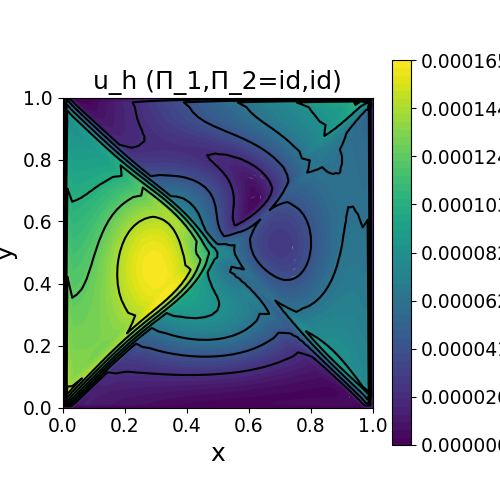}
\hspace{-2ex}
\includegraphics[width=0.22\textwidth, trim=5mm 0mm 0mm 0mm, clip]{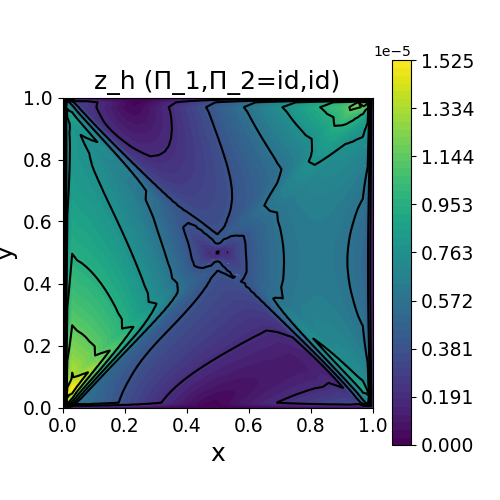}
\hfill
\includegraphics[width=0.27\textwidth]{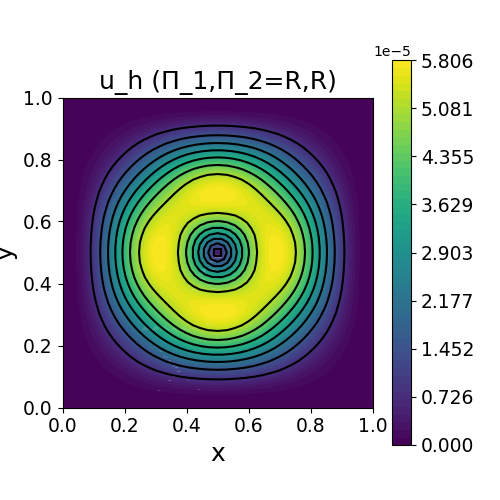}
\hspace{-2ex}
\includegraphics[width=0.22\textwidth, trim=5mm 0mm 0mm 0mm, clip]{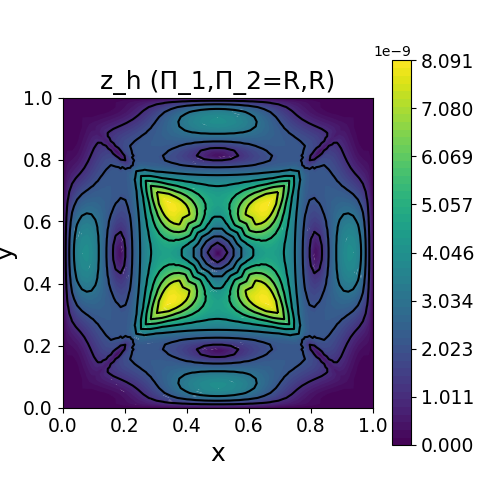}

\caption{\label{fig:ex1_images_br}Example 1: Discrete solutions $\vecb{u}_h$ (larger images) and $\vecb{z}_h$ (smaller images) for classical (left) and fully pressure-robust (right) Bernardi--Raugel method for \(\epsilon = 10^{-4}\), \(\nu = 10^{-3}\) and \(\alpha = 10^{-1}, 10^{-3}, 10^{-6}\) (from top to bottom).}
\end{figure}

\begin{figure}
\flushleft
\hspace{35ex} \fbox{classical scheme}

\vspace{-3ex}
\includegraphics[width=0.36\textwidth]{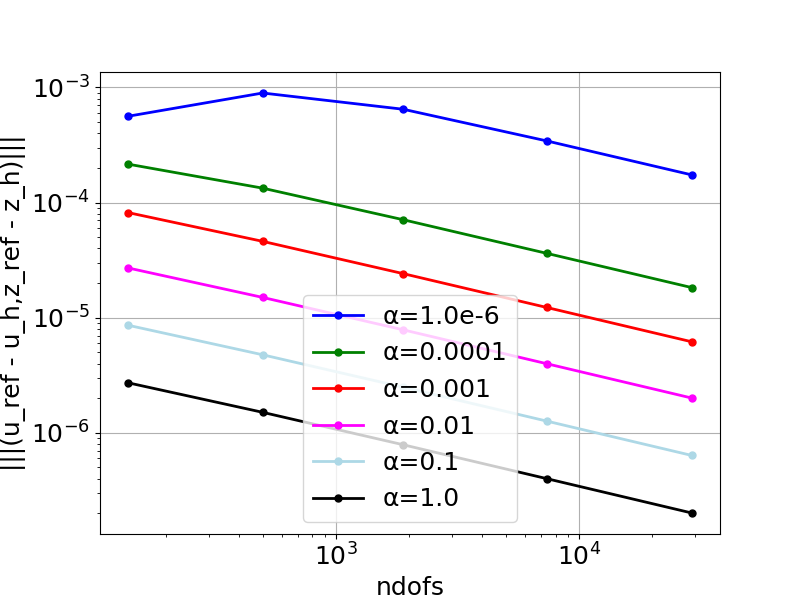}
\includegraphics[width=0.31\textwidth]{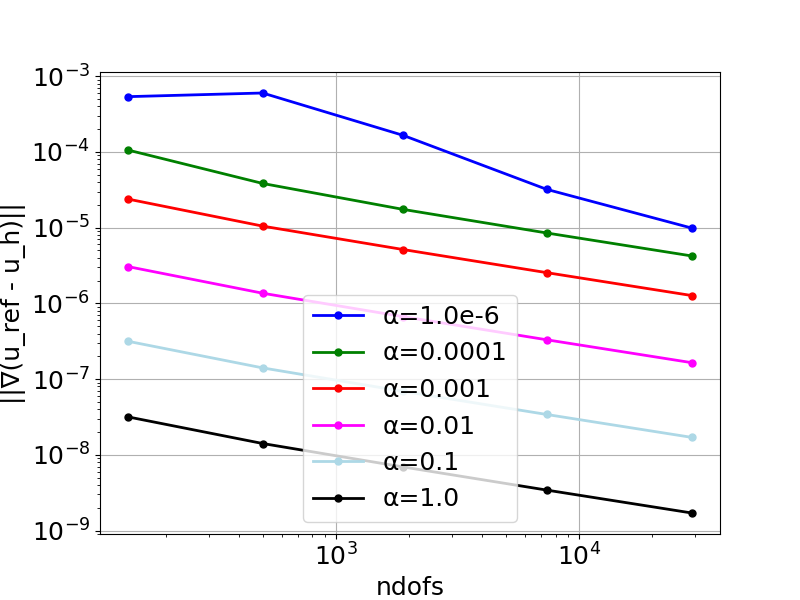}
\includegraphics[width=0.31\textwidth]{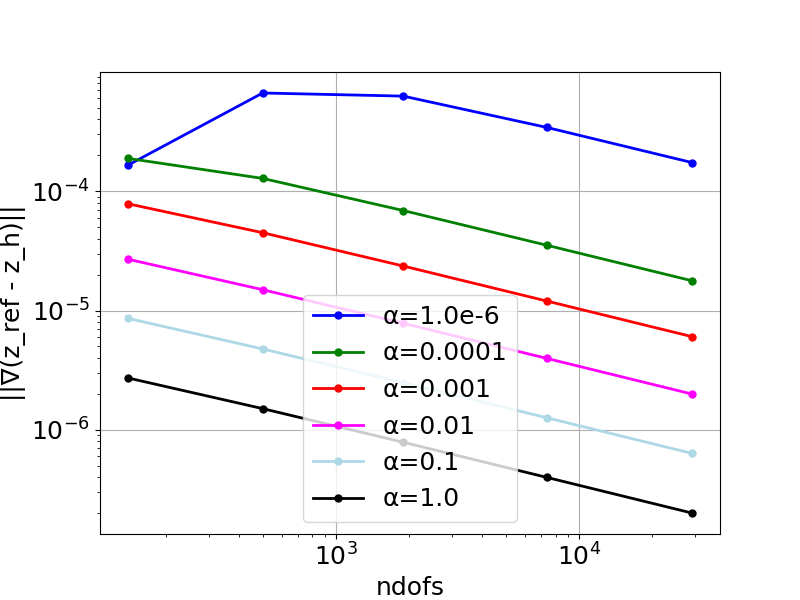}

\vspace{2ex}
\hspace{35ex} \fbox{partially p-robust scheme}

\vspace{-3ex}
\includegraphics[width=0.36\textwidth]{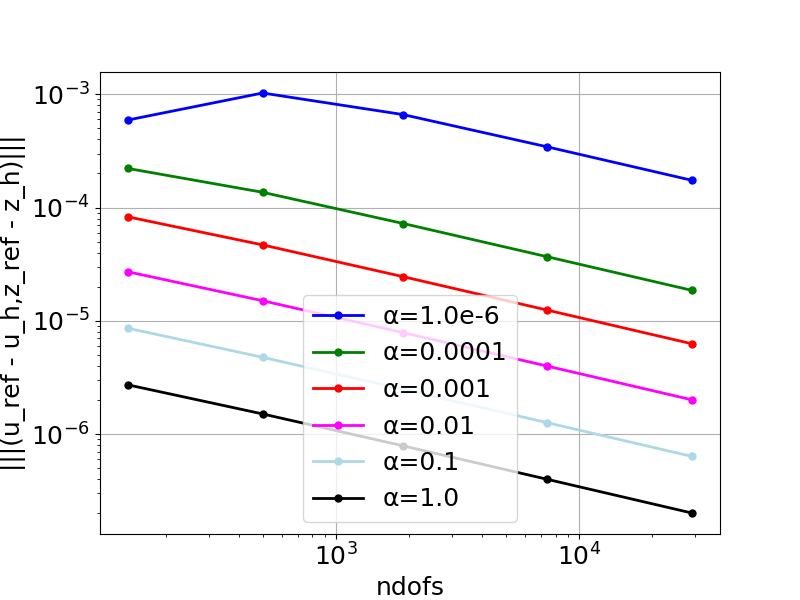}
\includegraphics[width=0.31\textwidth]{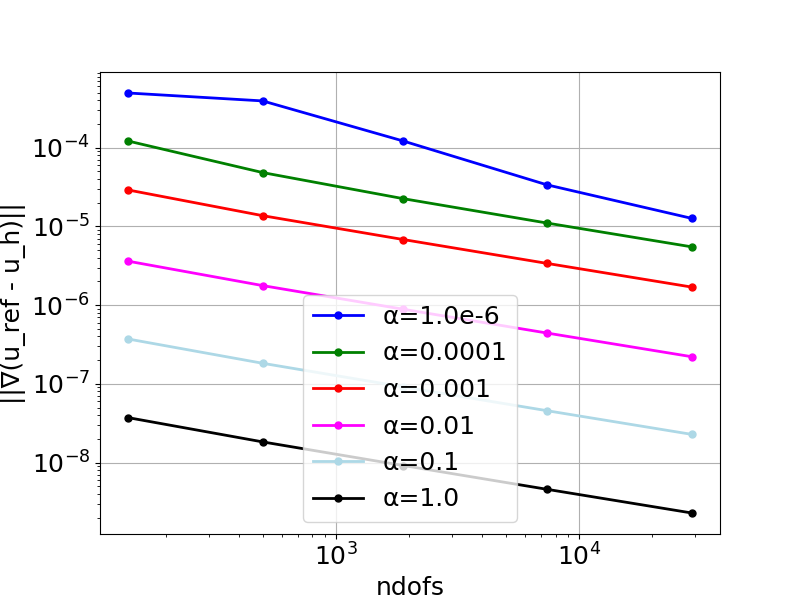}
\includegraphics[width=0.31\textwidth]{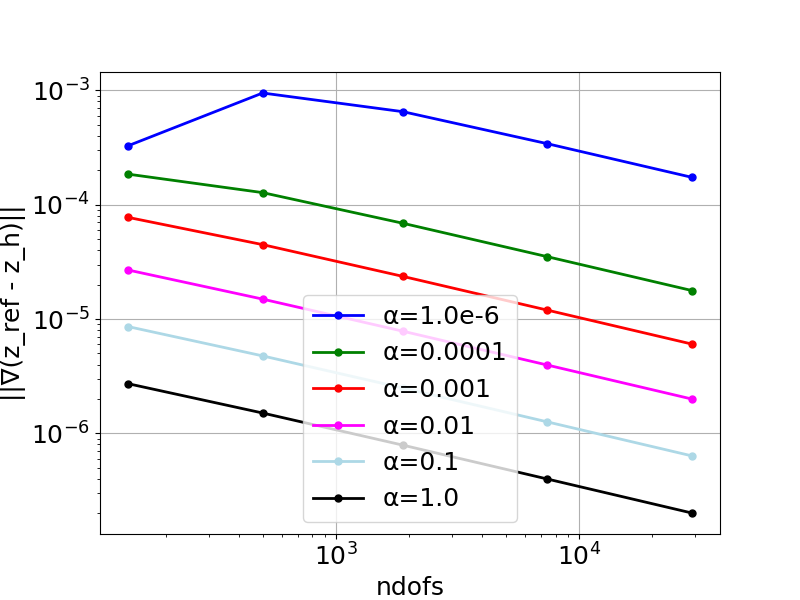}

\vspace{2ex}
\hspace{35ex} \fbox{fully p-robust scheme}

\vspace{-3ex}
\includegraphics[width=0.36\textwidth]{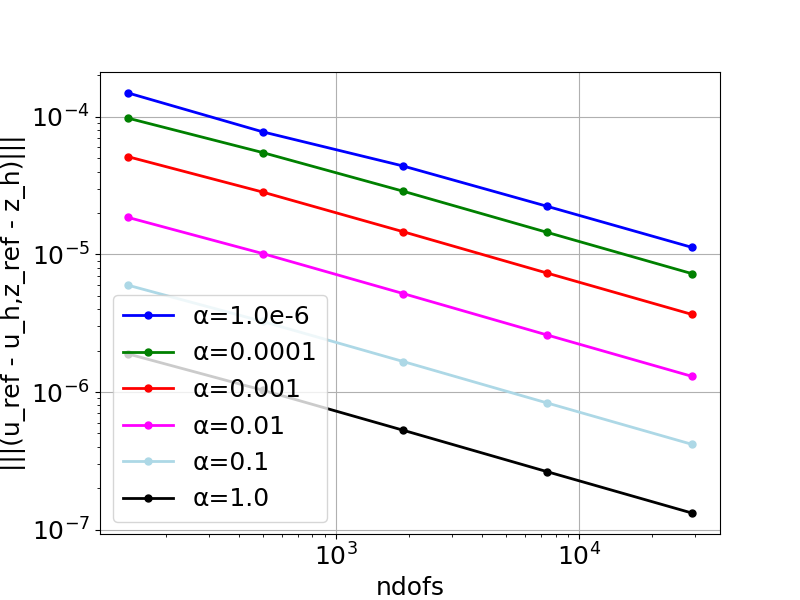}
\includegraphics[width=0.31\textwidth]{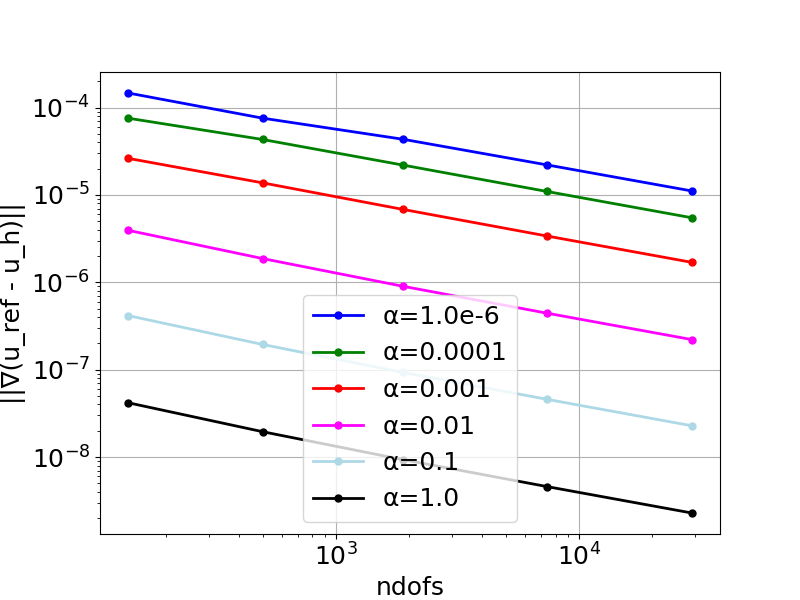}
\includegraphics[width=0.31\textwidth]{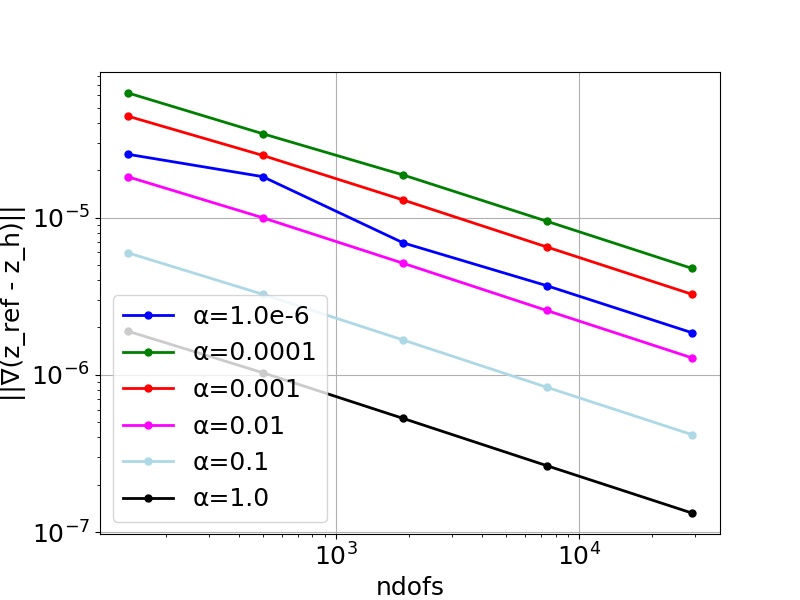}
\caption{\label{fig:ex1_convhist_mu0}Example 1: Convergence histories for the classical (top row),  partially pressure-robust (middle row) and fully pressure-robust (bottom row) Bernardi--Raugel methods for \(\epsilon = 10^{-4}\), \(\nu = 1\) and various choices of \(\alpha\). The first, second and third column depict the total energy error, velocity error, and the control error, respectively.}
\end{figure}

\begin{figure}

\flushleft
\hspace{35ex} \fbox{classical scheme}

\vspace{-3ex}
\includegraphics[width=0.36\textwidth]{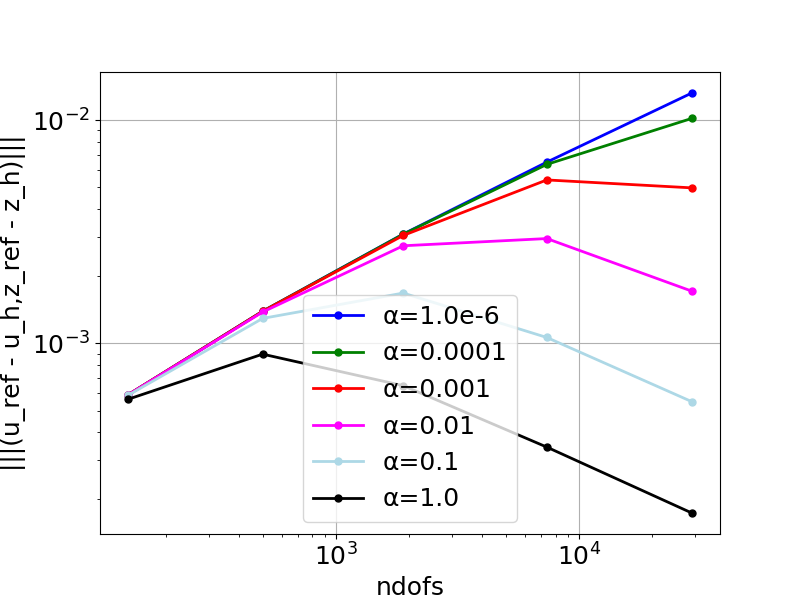}
\includegraphics[width=0.31\textwidth]{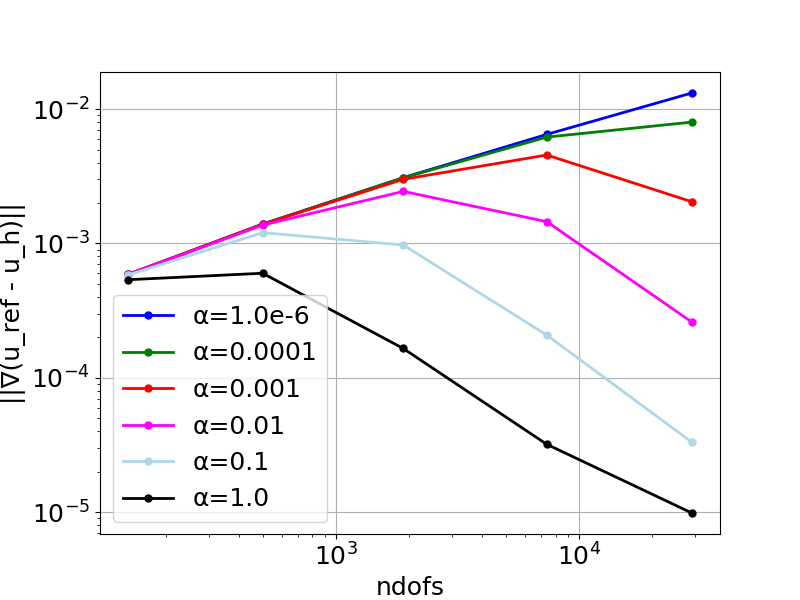}
\includegraphics[width=0.31\textwidth]{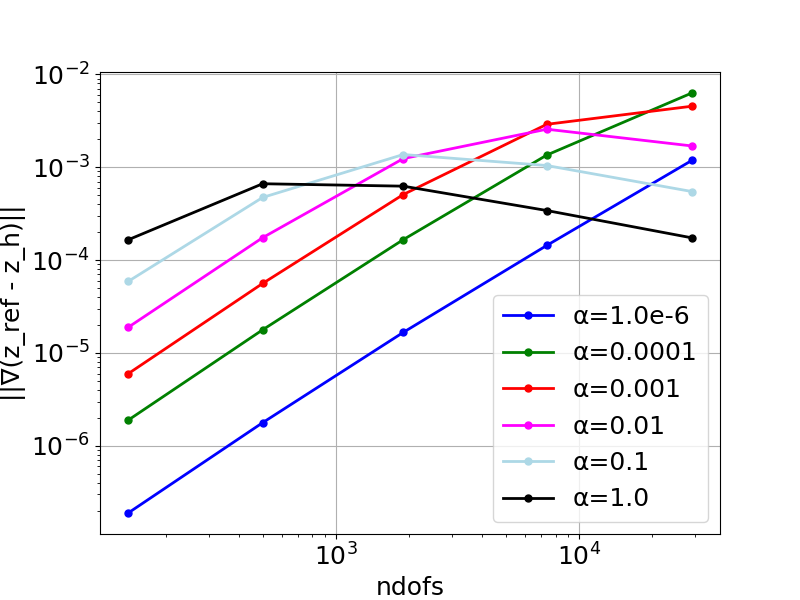}

\vspace{2ex}
\hspace{35ex} \fbox{partially p-robust scheme}

\vspace{-3ex}
\includegraphics[width=0.36\textwidth]{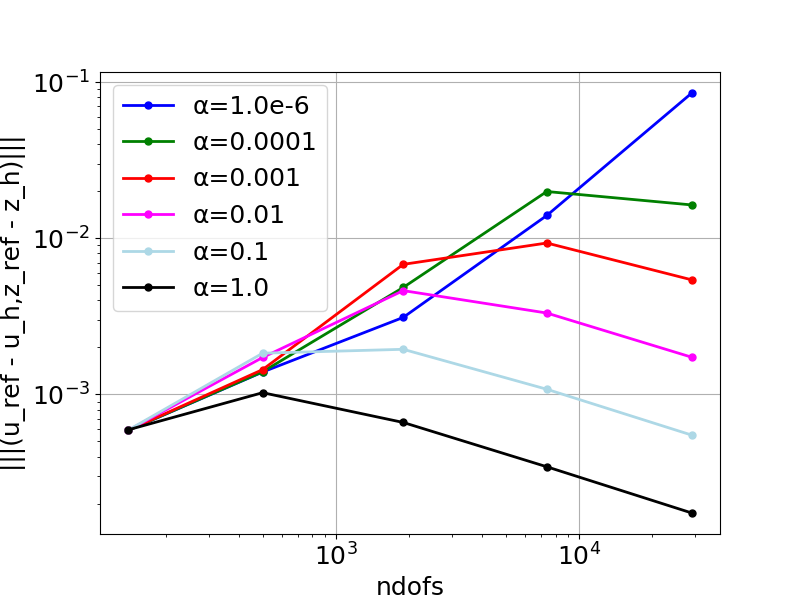}
\includegraphics[width=0.31\textwidth]{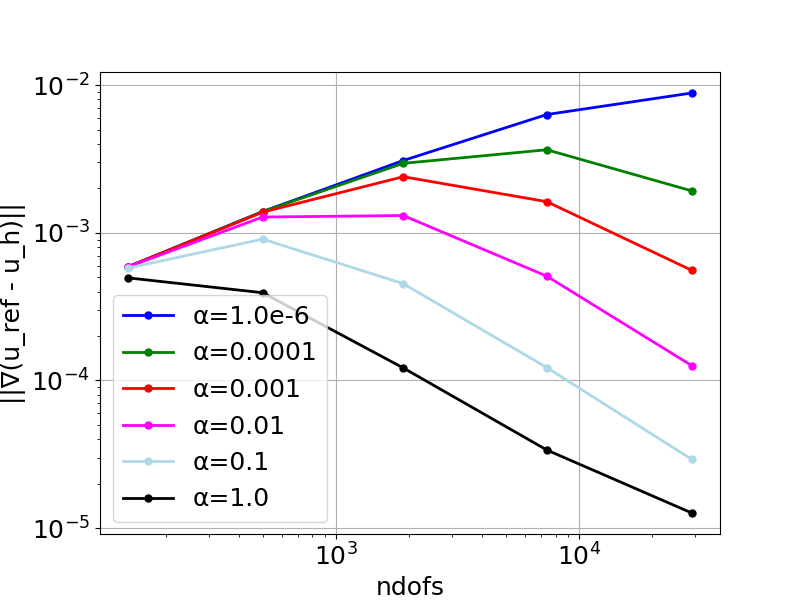}
\includegraphics[width=0.31\textwidth]{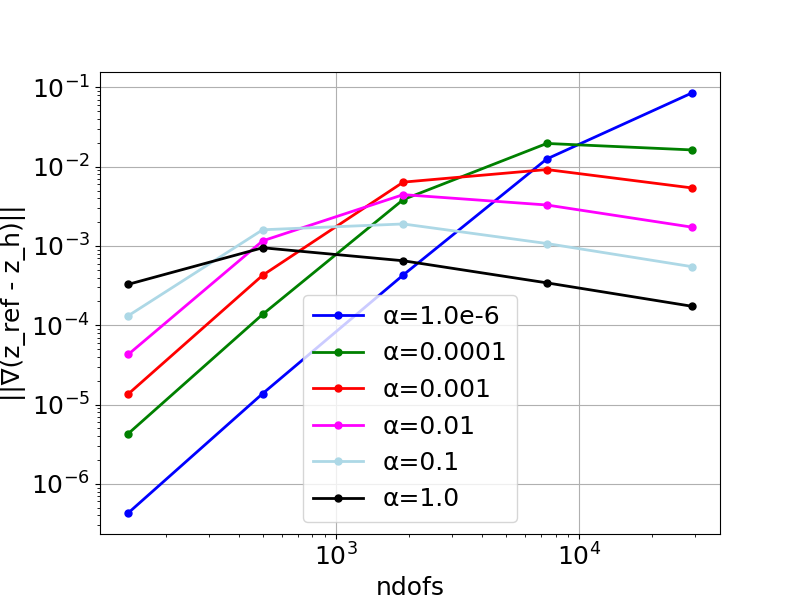}

\vspace{2ex}
\hspace{35ex} \fbox{fully p-robust scheme}

\vspace{-3ex}
\includegraphics[width=0.36\textwidth]{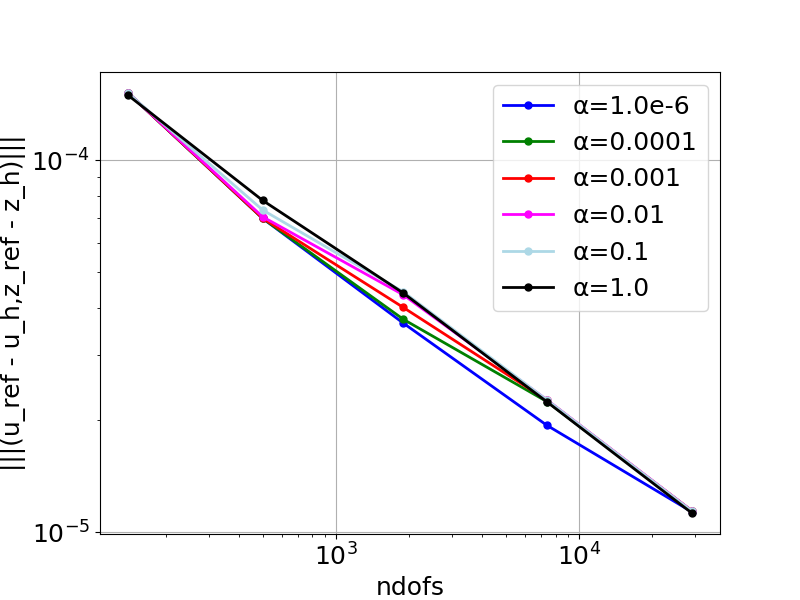}
\includegraphics[width=0.31\textwidth]{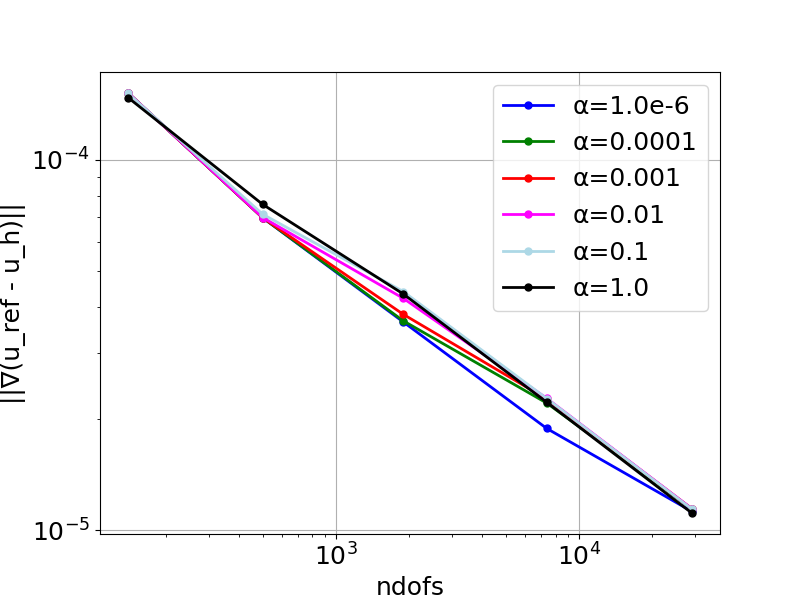}
\includegraphics[width=0.31\textwidth]{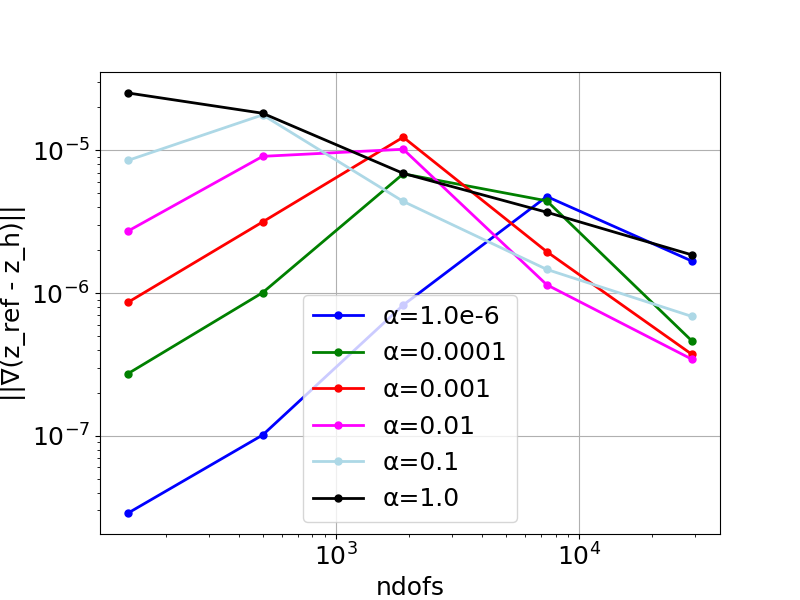}
\caption{\label{fig:ex1_convhist_mu3}Example 1: Convergence histories for the classical (top row),  partially pressure-robust (middle row) and fully pressure-robust (bottom row) Bernardi--Raugel methods for \(\epsilon = 10^{-4}\), \(\nu = 10^{-3}\) and various choices of \(\alpha\). The first, second and third column depict the total energy error, velocity error, and the control error, respectively.}
\end{figure}

\begin{figure}
\includegraphics[width=0.4\textwidth]{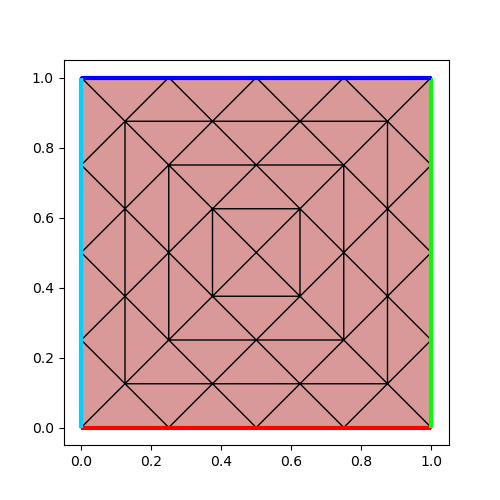}
\includegraphics[width=0.4\textwidth]{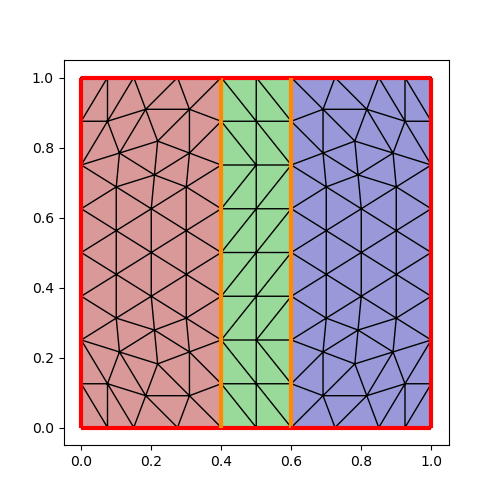}
\caption{\label{fig:grids}Coarsest grids used in Example~1 (left) and Example~2 (right).}
\end{figure}

\subsection{Example 1}
This example studies the prescribed polynomial solution
\begin{align*}
  \vecb{u}(x,y) = \mathrm{curl} (x^4(x-1)^4y^4(y-1)^4)
\end{align*}
of the Stokes problem \(- \nu \Delta \vecb{u} = \vecb{q}\) on the unit square \(\Omega = (0,1)^2\). The solution satisfies \(\vecb{q} := - \nu \Delta \vecb{u} \in H^1_0(\Omega)\)
and therefore \((\vecb{q}, \vecb{u})\) minimizes the objective functional for \(\alpha = 0\), \(\vecb{f} = 0\) and \(\vecb{u}^d := \vecb{u}\).
For \(\alpha > 0\), the exact minimizer is unknown and a reference solution is computed on a very fine grid with the second-order divergence-free Scott-Vogelius finite element method.
To study pressure-robustness, we perturb the data with some irrotational gradient field
\begin{align*}
   \vecb{u}^d(x,y) := \vecb{u}(x,y) + \epsilon\nabla(\cos(x)\sin(y))
\end{align*}
for different choices of \(\epsilon \geq 0\). From the analysis it is
clear that any divergence-free scheme like the Scott-Vogelius element
ignores the irrotational part in the data and therefore is independent
of \(\epsilon\). Figure~\ref{fig:ex1_refsolutions} shows some
reference solutions for fixed \(\nu = 1\) and different choices of \(\alpha\).

Figures~\ref{fig:ex1_convhist_mu0} shows the convergence history for \(\nu = 1\), where all methods under consideration perform very similar. Only in the case \(\alpha = 10^{-6}\) and \(\alpha = 10^{-4}\) the errors for \(\vecb{q}\) of the classical method and the partially pressure-robust method are worse than the error of the pressure-robust method. Also the velocity error behaves a bit suboptimal pre-asymptotically in these cases.

Figure~\ref{fig:ex1_images_br} shows some discrete solutions of the
classical scheme and the fully pressure-robust scheme with \(\epsilon
= 10^{-4}\) and moderate viscosity \(\nu = 10^{-3}\) and different
choices of \(\alpha\). For smaller \(\alpha\) significant errors in
\(\vecb{u}\) and \(\vecb{z}\) can be seen, while for \(\alpha = 0.1\)
only \(\vecb{z}\) looks heavily distorted in the classical scheme. 
For the smallest \(\alpha = 10^{-6}\) also the \(\vecb{z}\) of the
full robust scheme looks different than the reference solution, but at
least it looks symmetric and the magnitude is matched. Images for the
partially pressure-robust scheme are not presented, but look very
similar to those for the classical one.
The observations are also inline with the convergence histories in Figure~\ref{fig:ex1_convhist_mu3}. Here the situation for the classical and also the partially pressure-robust method is dramatically different. The error in the energy norm even diverges pre-asymptotically (the effect scales with \(\alpha^{-1}\) and becomes more pronounced for smaller \(\nu\). Only the fully pressure-robust method shows optimal convergence rates in the full range of tested parameters.

\subsection{Example 2}
Consider a unit square \(\Omega = (0,1)^2 = \Omega_C \cup
\Omega_F \cup \Omega_O\) decomposed into a control region \(\Omega_C = (0,2/5)
\times (0,1)\), a free region \(\Omega_F = (2/5,3/5) \times (0,1)\)
and an observation region \(\Omega_O = (3/5,1) \times (0,1)\). The
right part of Figure~\ref{fig:grids} shows a coarse triangulation
where these regions are marked with red, green and blue color in the
mentioned order.

\begin{figure}
\flushleft
\begin{tabular}{lll}
\fbox{$\alpha = 10^{-1}$, classical} &
\hspace{24ex}					   &
\fbox{$\alpha = 10^{-1}$, fully p-robust} 
\end{tabular}
\includegraphics[width=0.27\textwidth]{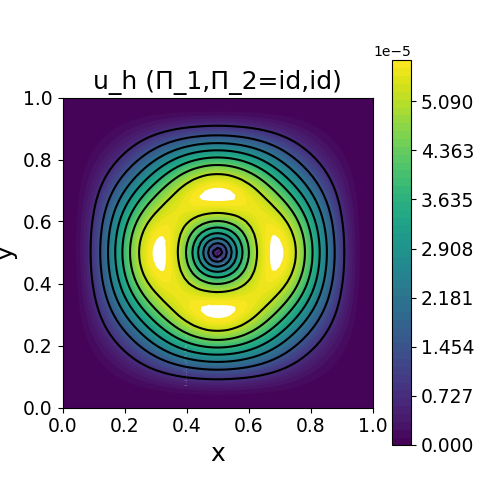}
\hspace{-2ex}
\includegraphics[width=0.22\textwidth, trim=5mm 0mm 0mm 0mm, clip]{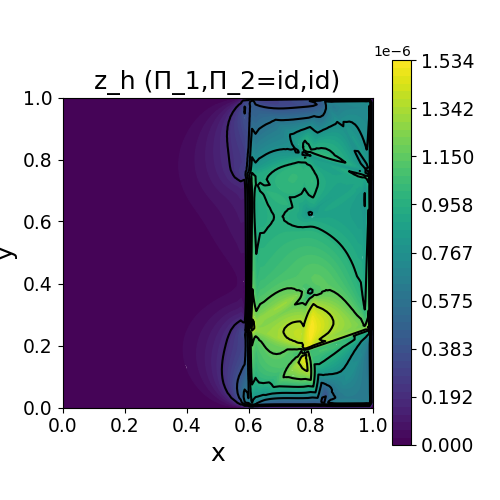}
\hfill
\includegraphics[width=0.27\textwidth]{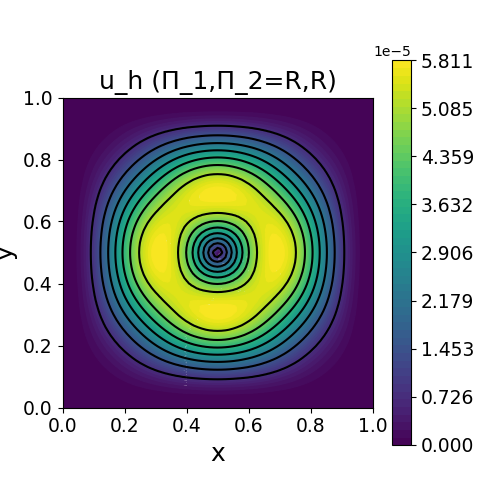}
\hspace{-2ex}
\includegraphics[width=0.22\textwidth, trim=5mm 0mm 0mm 0mm, clip]{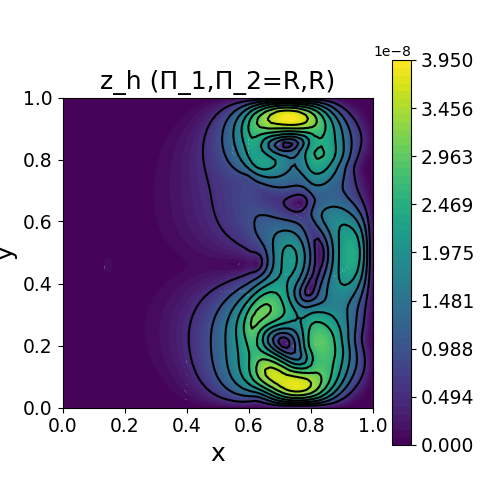}

\begin{tabular}{lll}
\fbox{$\alpha = 10^{-3}$, classical} &
\hspace{24ex}					   &
\fbox{$\alpha = 10^{-3}$, fully p-robust} 
\end{tabular}
\includegraphics[width=0.27\textwidth]{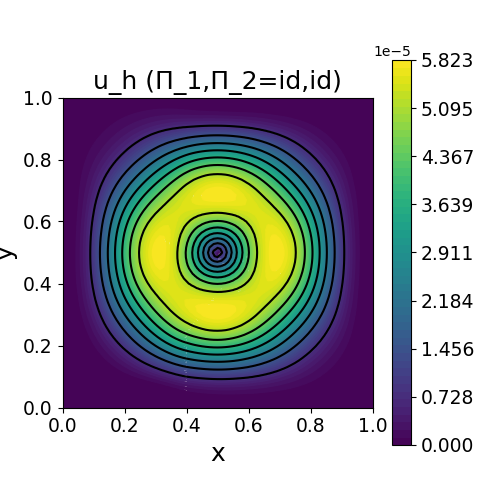}
\hspace{-2ex}
\includegraphics[width=0.22\textwidth, trim=5mm 0mm 0mm 0mm, clip]{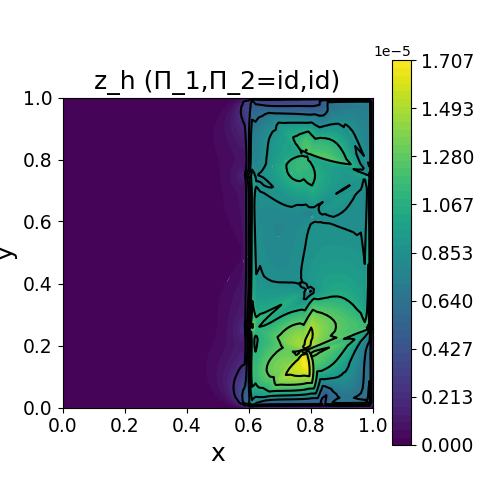}
\hfill
\includegraphics[width=0.27\textwidth]{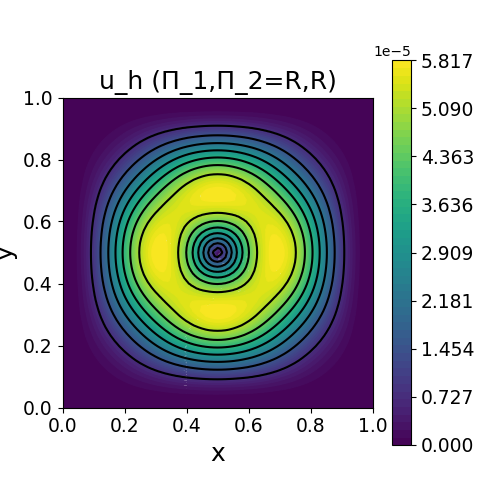}
\hspace{-2ex}
\includegraphics[width=0.22\textwidth, trim=5mm 0mm 0mm 0mm, clip]{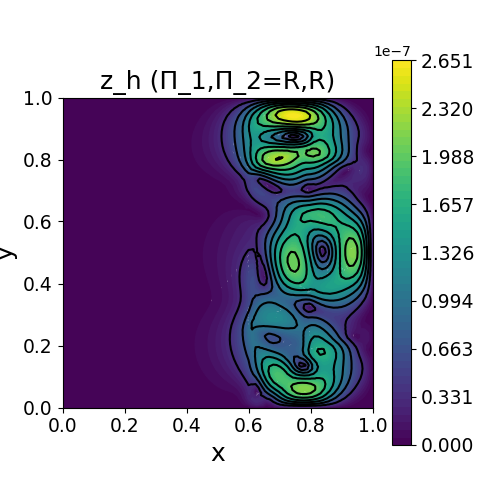}

\begin{tabular}{lll}
\fbox{$\alpha = 10^{-6}$, classical} &
\hspace{24ex}					   &
\fbox{$\alpha = 10^{-6}$, fully p-robust} 
\end{tabular}
\includegraphics[width=0.27\textwidth]{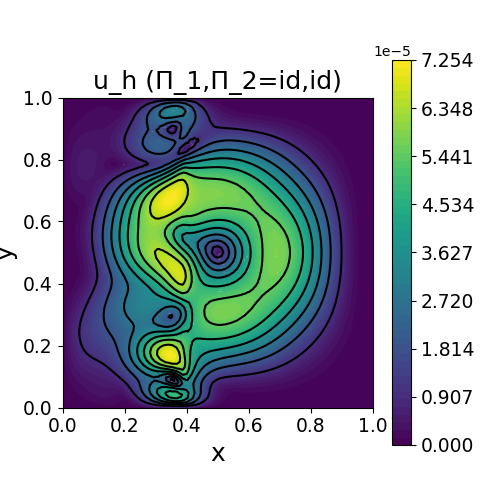}
\hspace{-2ex}
\includegraphics[width=0.22\textwidth, trim=5mm 0mm 0mm 0mm, clip]{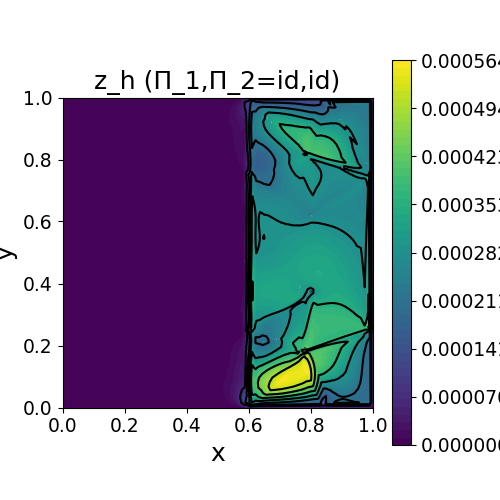}
\hfill
\includegraphics[width=0.27\textwidth]{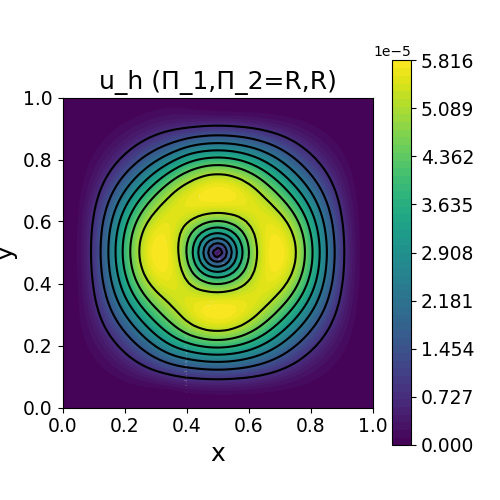}
\hspace{-2ex}
\includegraphics[width=0.22\textwidth, trim=5mm 0mm 0mm 0mm, clip]{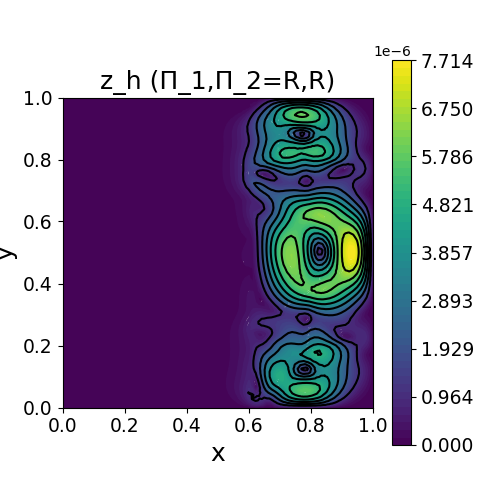}

\caption{\label{fig:ex2_images_br}Example : Discrete solutions $\vecb{u}_h$ (larger images) and $\vecb{z}_h$ (smaller images) for classical (left) and fully pressure-robust (right) Bernardi--Raugel method for \(\epsilon = 10^{-4}\), \(\nu = 10^{-3}\) and \(\alpha = 10^{-1}, 10^{-3}, 10^{-6}\) (from top to bottom).}
\end{figure}

\begin{figure}
\flushleft

\vspace{2ex}
\hspace{35ex} \fbox{classical scheme}

\vspace{-3ex}
\includegraphics[width=0.36\textwidth]{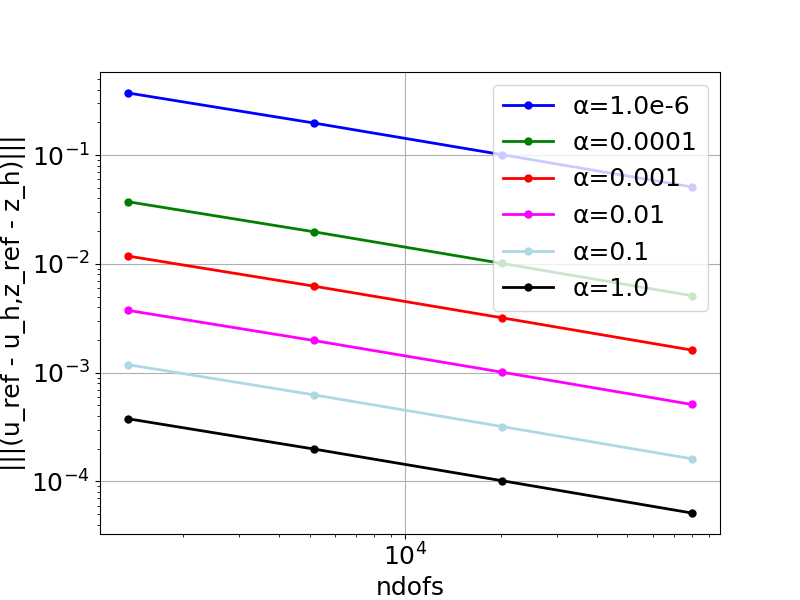}
\includegraphics[width=0.31\textwidth]{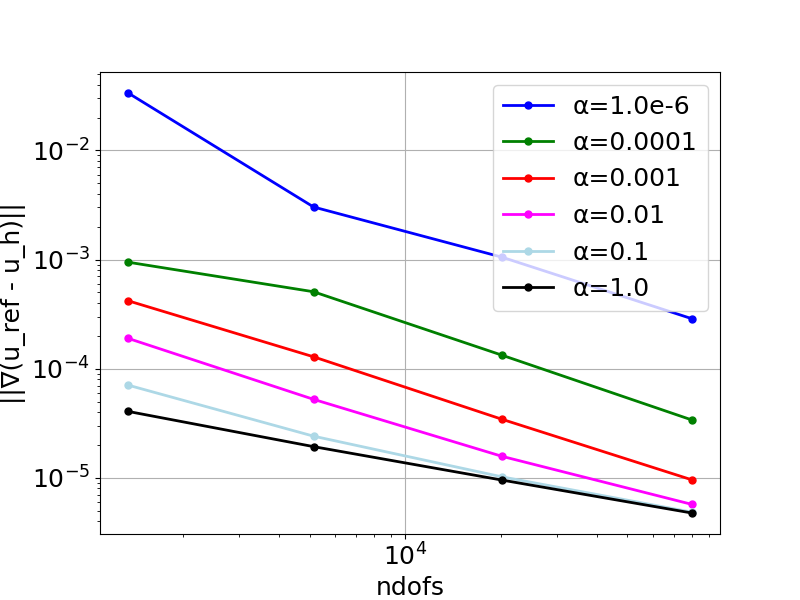}
\includegraphics[width=0.31\textwidth]{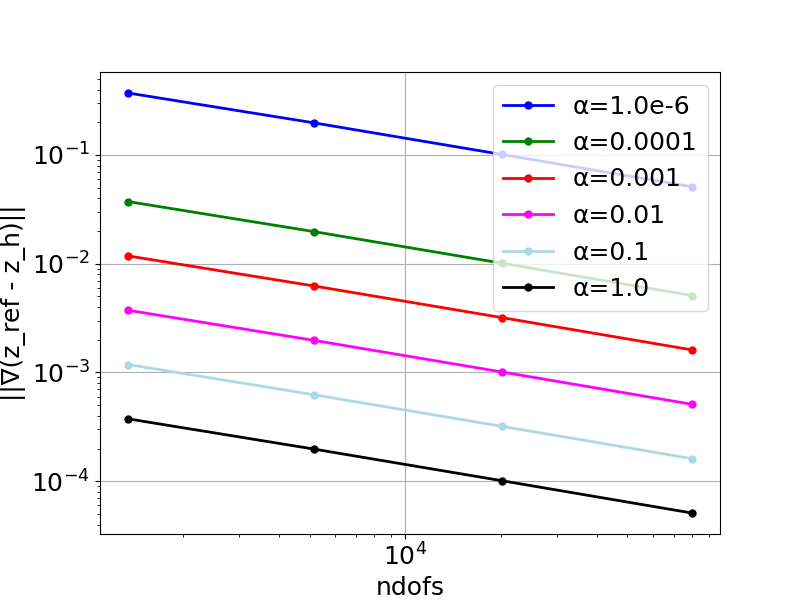}

\vspace{2ex}
\hspace{35ex} \fbox{partially p-robust scheme}

\vspace{-3ex}
\includegraphics[width=0.36\textwidth]{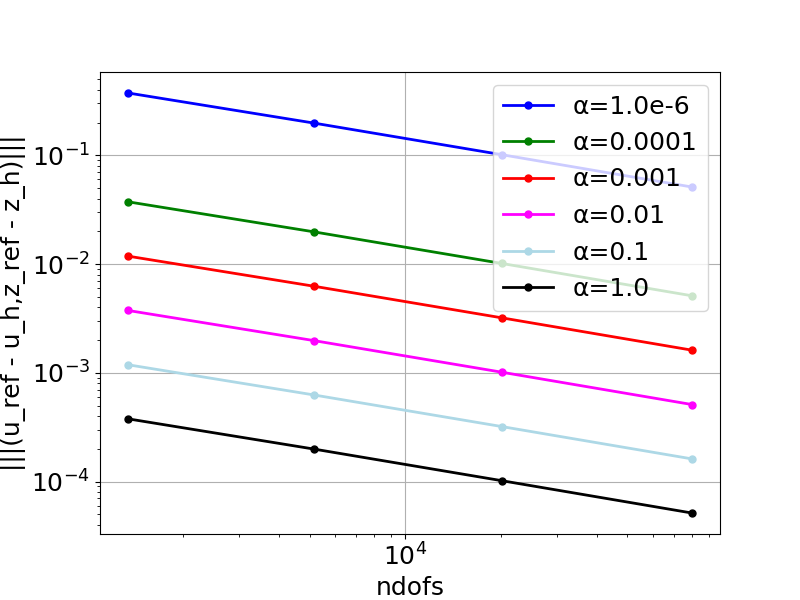}
\includegraphics[width=0.31\textwidth]{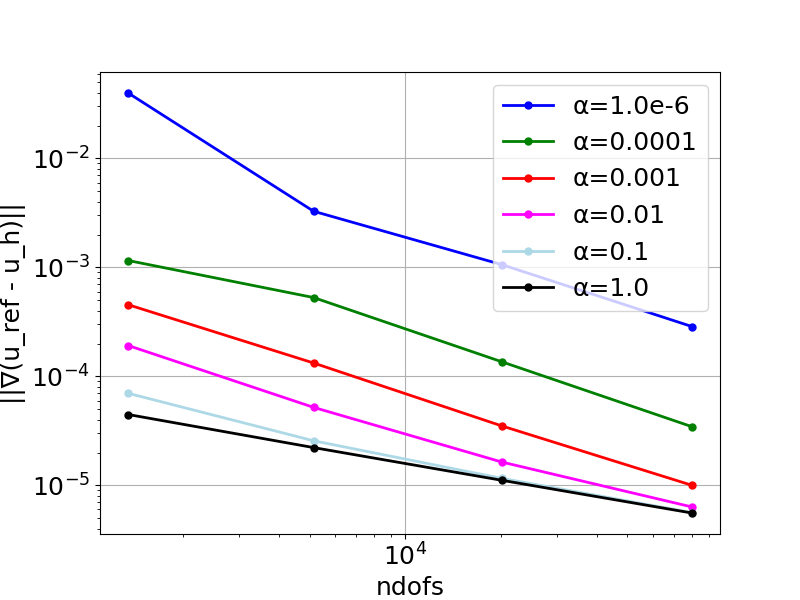}
\includegraphics[width=0.31\textwidth]{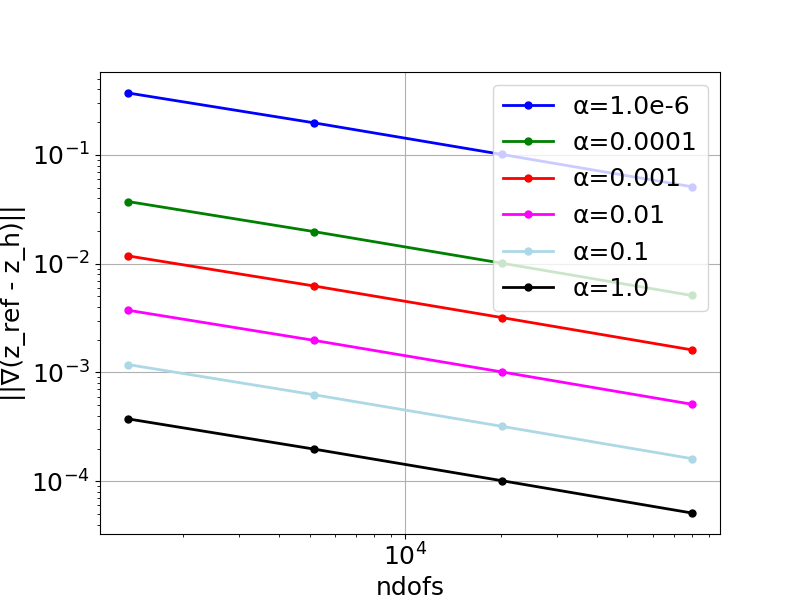}

\vspace{2ex}
\hspace{35ex} \fbox{fully p-robust scheme}

\vspace{-3ex}
\includegraphics[width=0.36\textwidth]{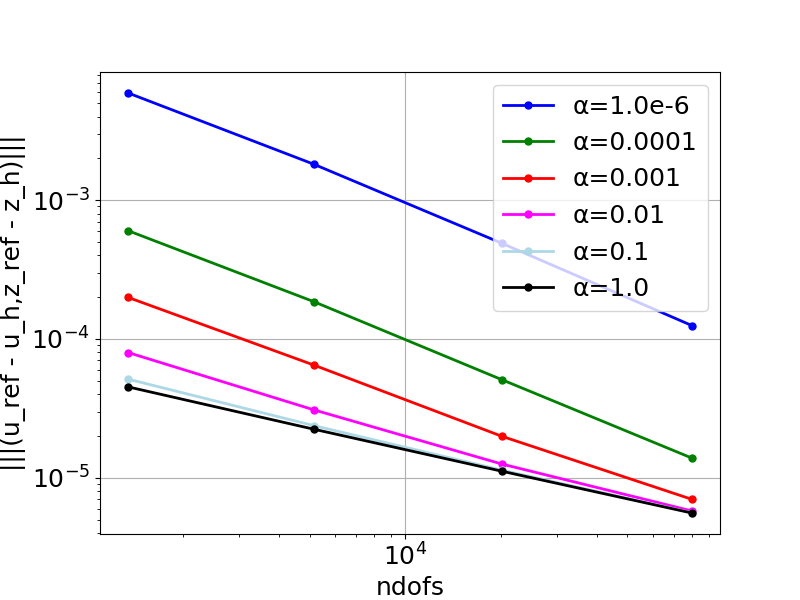}
\includegraphics[width=0.31\textwidth]{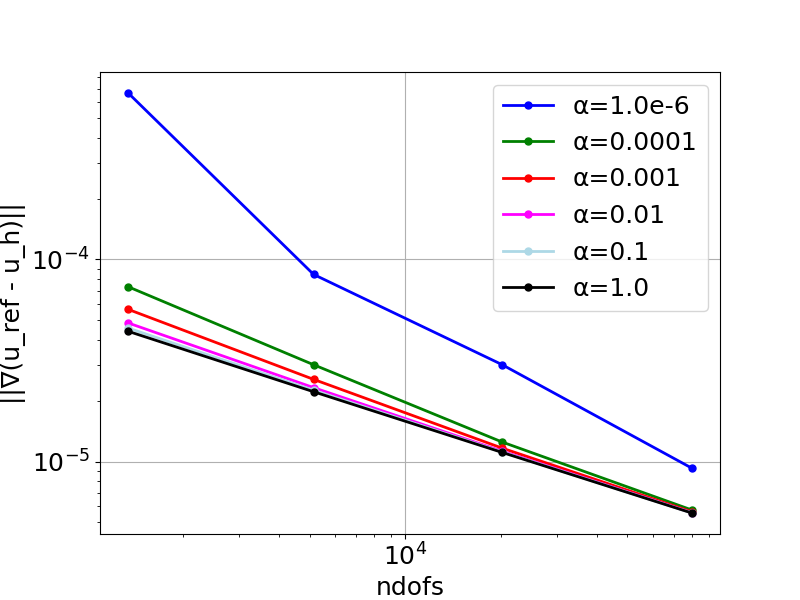}
\includegraphics[width=0.31\textwidth]{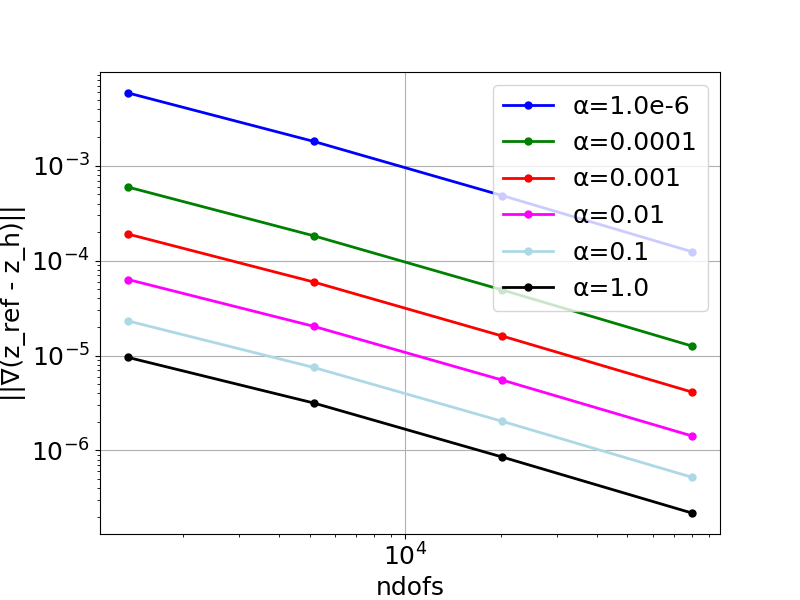}
\caption{\label{fig:ex2_convhist_mu3}Example 2: Convergence histories for the classical (top row),  partially pressure-robust (middle row) and fully pressure-robust (bottom row) Bernardi--Raugel methods for \(\epsilon = 10^{-4}\), \(\nu = 10^{-3}\) and various choices of \(\alpha\). The first, second and third column depict the total energy error, velocity error, and the control error, respectively.}
\end{figure}

By straightforward arguments, the optimal control is reformulated into
\begin{align*}
  \min_{(\vecb{q}_h,\vecb{u}_h,p_h) \in \vecb{Q} \times \vecb{V}_h\times Q_h}
    &\;\frac{1}{2}\|\Pi_1 \vecb{u}_h-\vecb{u}^{\rm{d}}\|_{\vecb{L}^2(\Omega_O)}^2
    + \frac{\alpha}{2}\|\vecb{q}_h\|_{\vecb{L}^2(\Omega_C)}^2\\
    \text{s.t.}&\; \left\{
      \begin{aligned}
        \nu (\nabla \vecb{u}_h,\nabla \vecb{\varphi}_h) + (p_h,\Div
        \vecb{\varphi}_h) &=   (\vecb{f},\Pi_2\vecb{\varphi}_h) + (\vecb{q}_h,\Pi_2\vecb{\varphi}_h)_{\Omega_C} &
        \forall&\vecb{\varphi}_h\in \vecb{V}_h,\\
        (\Div\vecb{u}_h, \psi_h) &= 0 & \forall&\psi_h \in Q_h.
      \end{aligned}\right.
  \end{align*}
The optimization problem is equivalent to searching for a
solution \((\vecb{u}_h, \vecb{z}_h, p_h , \lambda_h) \in \vecb{V}_h \times \vecb{V}_h \times Q_h \times Q_h\) of
  \begin{equation*}
\begin{aligned}
  \nu (\nabla \vecb{u}_h, \nabla \vecb{\varphi}_h) - (\Div \vecb{\varphi}_h,p_h)& = (\vecb{f}, \Pi_2 \vecb{\varphi}_h) - (\alpha^{-1/2}\Pi_2 \vecb{z}_h, \Pi_2 \vecb{\varphi}_h)_{\Omega_C}& \forall &\vecb{\varphi}_h\in \vecb{V}_h,\\
  (\Div \vecb{u}_h , q_h) & = 0& \forall&q_h \in Q_h,\\
  \nu (\nabla \vecb{\varphi}_h, \nabla \vecb{z}_h) + (\Div
  \vecb{\varphi}_h,\lambda_h) & = \alpha^{-1/2} (\Pi_1 \vecb{u}_h - \vecb{u}^{\rm{d}}, \Pi_1 \vecb{\varphi}_h)_{\Omega_O}& \forall &\vecb{\varphi}_h\in \vecb{V}_h,\\
  (\Div \vecb{z}_h , q_h) & = 0& \forall&q_h \in Q_h.
\end{aligned}
\end{equation*}

This example employs the same data from Example~1, but with a perturbation that really is orthogonal on divergence-free functions when integrated over \(\Omega_O\), i.e.
\begin{align*}
   \vecb{u}^d(x,y) := \vecb{u}(x,y) + \epsilon\nabla(\sin(x-0.6)\cos(y)).
 \end{align*}
Moreover, we prescribe \(f := -\mu\Delta \vecb{u}\) such that \((\vecb{0}, \vecb{u})\) is the minimizer of the objective functional for \(\alpha = 0\), so this time the control should be close to \(\vecb{z} = \vecb{0}\). For \(\alpha > 0\), the exact minimizer is once again approximated on a very fine grid with the second-order divergence-free Scott-Vogelius finite element method.

As in the previous example Figure~\ref{fig:ex2_images_br} depicts some discrete solutions of the classical scheme and the fully robust scheme for \(\nu = 10^{-3}\) and different choices of \(\alpha\). The solution \(\vecb{z}_h\) of the fully robust method is about two orders of magnitudes closer to \(\vecb{z} = 0\) than the classical scheme. This is also supported by the convergence histories in Figure~\ref{fig:ex2_convhist_mu3}. For very small \(\alpha\) the errors for the fully robust scheme are also about two order of magnitudes better than the errors of the classical scheme and also seem to converge faster. This may be explained by the velocity error \(\vecb{u}_h\) that is almost independent of \(\alpha\) and only gets larger for very small \(\alpha\). This might be caused by the higher-order term in Lemma~\ref{lem:fullyrobust}.

\section*{Acknowledgments}
We thank Alexander Linke for many pleasant and fruitful
discussions. W.~Wollner acknowledges funding by the Deutsche
Forschungsgemeinschaft (DFG, German Research Foundation) -- Projektnummer 392587580 -- SPP 1748